\documentclass[a4paper]{amsart}

\calclayout

\usepackage[utf8]{inputenc}
\usepackage[english, swedish]{babel}
\usepackage{amsmath, bm}
\usepackage{graphicx}
 \usepackage{todonotes}

\usepackage{amsfonts}
\usepackage{siunitx}
\usepackage{float}
\usepackage{rotating}
\usepackage[section]{placeins}
\usepackage{ stmaryrd }
\usepackage{ marvosym }
\usepackage{mathrsfs}
\usepackage[all,cmtip]{xy}
\usepackage{adjustbox}

\usepackage{amssymb,microtype}

\usepackage{amsthm}

\usepackage{thmtools, thm-restate}

\usepackage{tikz}
\usetikzlibrary{babel}
\usepackage{sseq}

\usepackage{microtype}

\usepackage{tikz-cd}

\usepackage[english, status=draft]{fixme}

\usepackage[toc,title]{appendix}

\usepackage[colorlinks=false]{hyperref}

\usepackage{amscd}

\fxusetheme{color}

\usetikzlibrary{matrix,arrows,decorations}

\DeclareMathOperator{\Ab}{Ab}

\DeclareMathOperator{\Gal}{Gal}
\DeclareMathOperator{\Tot}{Tot}

\DeclareMathOperator{\Free}{Free}
\DeclareMathOperator{\Gen}{Gen}

\DeclareMathOperator{\Ext}{Ext}

\DeclareMathOperator{\Ind}{Ind}

\DeclareMathOperator{\Cl}{Cl}

\DeclareMathOperator{\divis}{div}

\DeclareMathOperator{\Div}{Div}
\DeclareMathOperator{\rk}{rk}

\DeclareMathOperator{\Sh}{Sh}

\newcommand{\C}{\mathcal C}
\newcommand{\Ei}{\mathcal E}

\newcommand{\Ufl}{U_{\normalfont \textsf{fppf}}}
\newcommand{\fl}{{\normalfont \textsf{fppf}}}

\usepackage{amsmath,calligra,mathrsfs}

\DeclareMathOperator{\HOM}{\mathscr{H}\text{\kern -2.5pt {\emph{om}}}}

\DeclareMathOperator{\DIV}{\mathscr{D}\text{\kern -1pt {\emph{iv}}}\,}

\newcommand{\K}{\mathcal K}

\newcommand{\Hom}{{\operatorname{Hom}}}

\newcommand{\RHom}{R\Hom}

\newcommand{\RHOM}{R\kern -2pt\HOM}

\newcommand{\mmu}{\pmb{\mu}}
\newcommand{\legendre}[2]{\genfrac{(}{)}{}{}{#1}{#2}}

\newcommand{\Spec}{\normalfont \text{Spec }}

\newcommand{\GG}{\mathbb{G}}
\newcommand{\GmK}{\GG_{m,K}}
\newcommand{\OO}{{\mathcal O}}

\newcommand{\ZZ}{{\mathbb Z}}

\newcommand{\QQ}{{\mathbb Q}}
\newcommand{\RR}{{\mathbb R}}
\newcommand{\CC}{{\mathbb C}}
\newcommand{\LL}{{\mathbb L}}
\newcommand{\Li}{\mathcal{L}}

\newcommand{\cupp}{\smallsmile}

\makeatletter

\newcommand{\et}{\textsf{\'et}}

\newtheorem{thm}{Theorem}
\numberwithin{thm}{section}
\newtheorem{innercustomthm}{Theorem}
\newenvironment{customthm}[1]
  {\renewcommand\theinnercustomthm{#1}\innercustomthm}
  {\endinnercustomthm}

\newtheorem{prop}[thm]{Proposition}
\newtheorem{lem}[thm]{Lemma}
\newtheorem{cor}[thm]{Corollary}
\theoremstyle{definition}
\newtheorem{df}[thm]{Definition}
\newtheorem{ex}[thm]{Example}
\theoremstyle{remark}
\newtheorem{remark}[thm]{Remark}

\newcommand{\A}{\mathcal{A}}

\newcommand{\pr}{{\normalfont \text{pr}}}

\DeclareMathOperator{\Br}{Br}
\DeclareMathOperator{\Frac}{Frac}
\DeclareMathOperator{\inv}{inv}

\newcommand{\Sch}{{\normalfont\textsf{Sch}}}

\newcommand{\bb}{\mathfrak{b}}

\newcommand{\im}{{\normalfont\text{im}}\,}

\makeatletter
\newcommand{\pushright}[1]{\ifmeasuring@#1\else\omit\hfill$\displaystyle#1$\fi\ignorespaces}
\newcommand{\pushleft}[1]{\ifmeasuring@#1\else\omit$\displaystyle#1$\hfill\fi\ignorespaces}
\makeatother

\title{{The \'etale cohomology ring of a punctured arithmetic curve}}

\author{Eric Ahlqvist}
\address{Department of mathematics, Stockholms universitet, 
106 91 Stockholm, Sweden}
\email{eric.ahlqvist@math.su.se}

\author{Magnus Carlson}
\address{Institut für Mathematik, 
Johann Wolfgang Goethe-Universität, 
Robert-Mayer-Str. 6-8, 
D-60325 Frankfurt am Main, 
Germany}
\email{carlson@math.uni-frankfurt.de}

\begin{document}

\selectlanguage{english}

\begin{abstract}
We compute the cohomology ring $H^*(U,\mathbb{Z}/n\mathbb{Z})$ for $U=X\setminus S$ where $X$ is the spectrum of the ring of integers of a number field $K$ and $S$ is a finite set of finite primes. As a consequence, we obtain an efficient way to compute presentations of $Q_2(G_S)$, where $G_S$ is Galois group of the maximal extension of $K$ unramified outside of a finite set of primes $S$, for varying $K$. This includes the following cases (for $p$ any prime dividing $n$): $\mu_p(\overline{K}) \not\subseteq K$; $S$ does not contain the primes above $p$; and $p=2$ with $K$ admitting real archimedean places. We also show how to recover the classical reciprocity law of the Legendre symbol from the graded commutativity of the cup product.
\end{abstract}

\thanks{The first author was supported by the Swedish Research Council 2015-05554 and the Knut and Alice Wallenberg Foundation 2021.0279. The second author was sponsored by the Knut and Alice Wallenberg Foundation 2017.0400}

\maketitle

\setcounter{tocdepth}{1}
\tableofcontents

\section{Introduction}
Let $K$ be a number field with ring of integers $\OO_K$ and let $X=\Spec \OO_K$. In this paper we compute the \'etale cohomology ring $H^*(U,\ZZ/n\ZZ)$ where $U\subseteq X$ is an open subscheme. This generalizes the results in \cite{Ahlqvist--Carlson-cup} not only by going from $X$ to an open subscheme $U$, but also in the sense that it covers the case when $K$ has a real place, $n$ is even, and $U=X$. Our results also generalizes those of McCallum--Sharifi: in \cite{McCallum--Sharifi} they compute the cup product
$H^1(U,\mu_n) \otimes H^1(U,\mu_n)  \to H^2(U,\mu_n^{\otimes 2})$ when $K$ contains the $n$th roots of unity and $U = X \setminus S$ for $S$ a finite set of places containing the places above $n$ (note that since $S$ contains the places above $n$ and $K$ contains the $n$th roots of unity, $\mu_n \cong \ZZ/n\ZZ$). The computations we make in this article gives the cup product for an arbitrary finite set of non-archimedean places $S$, and we do not assume that $K$ contains the $n$th roots of unity.

The ring $H^*(U,\ZZ/n\ZZ)$ holds non-trivial arithmetic information about $K$. One illustration of this is the following example (see Section \ref{sec:examples}) first observed by Morishita: Let $U=\Spec\ZZ\setminus\{p,q\}$ where $p$ and $q$ are primes which are both $1$ (mod $4$).
Taking rings of integers of the quadratic extensions $\QQ(\sqrt{p})$ and $\QQ(\sqrt{q})$, we get two elements $x_p,x_q\in H^1(U,\ZZ/2\ZZ)$ and the element $x_p\cupp x_q$ is completely determined by the Legendre symbol $\legendre{p}{q}$ and vice versa. In particular, $x_p\cupp x_q=0$ if and only if $\legendre{p}{q}=1$.

We will now state the main results of this paper. Let $I_K$ be the group of non-complex idèles of $K$ and let $C_S(K)$ denote the $S$-idèle class group. We begin with stating the computation of the étale cohomology groups.

\begin{customthm}{\ref{thm:coh-groups}}
Let $K_+$ be the set of totally positive elements in $K$ and $\Cl^+(K)$ the narrow class group of $K$. Furthermore, let $r$ be the number of real places of $K$ and for any abelian group $A$ we denote by $A^\sim=\Hom(A,\QQ/\ZZ)$, the Pontryagin dual of $A$. We then have
  \[
    H^i(X,\ZZ/n\ZZ)\cong
    \begin{cases}
      \ZZ/n\ZZ & i=0\,, \\
      (\Cl^+(K)/n\Cl^+(K))^\sim & i=1\,, \\
      (Z^1/B^1)^\sim & i=2\,,\\
      (\mu_n(K_+) \oplus (\RR^\times/n\RR^\times)^r)^\sim & i=3\,,\\
      ((\RR^\times/n\RR^\times)^r)^\sim & i\geq 4 \ .
    \end{cases}
  \]
where
  \[
      \begin{split}
        Z^1 & = \{(a,I)\in K_+^\times\oplus \Div U : \divis(a)I^n=1\}\,, \\
        B^1 & = \{(b^{-n},\divis(b)):b\in K^\times\}\,.
      \end{split}
  \]
For a finite, non-empty set of finite places $S$ and $U = X \setminus S$, we have that
  \[
  H^i(U,\ZZ/n\ZZ)=\begin{cases}
    \ZZ/n\ZZ & i=0\,, \\
    (C_S(K)/nC_S(K))^\sim & i=1\,, \\
    (C_S(K)[n])^\sim & i=2\,,\\
    ((\RR^\times/n\RR^\times)^r)^\sim & i\geq 3\,.
  \end{cases}
  \]
\end{customthm}

The following theorem describes the cup product structure in $H^*(U,\ZZ/n\ZZ)$ when $S$ contains at least one finite place. This is in some sense the only ``interesting'' cup product in this situation.

\begin{customthm}{\ref{thm:cup}}
Suppose that $S$ is a finite non-empty set of finite places and let $U = X \setminus S$. Let $y$ and $z$ be elements in $H^1(U,\ZZ/n\ZZ)$. Suppose that $y$ is represented by a cyclic extension $L/K$, unramified outside of $S$, together with a choice of generator $\sigma \in \Gal(L/K)$, and assume that $L/K$ has degree $d|n$. Then under the identifications $H^1(U,\ZZ/n\ZZ)\cong (C_S(K)/nC_S(K))^\sim$ and $H^2(U,\ZZ/n\ZZ)\cong (C_S(K)[n])^\sim$ we have that $y\cupp z\in (C_S(K)[n])^\sim$ satisfies the formula
  \[
    \langle y\cupp z,\alpha\rangle=\langle z,\alpha^{n^2/2d}N_{L|K}(\beta)^{n/d}\rangle
  \]
where $\langle-,-\rangle$ is the evaluation map, $\beta$ is an element in $I_L$ such that $\alpha^{n/d}=t\beta/\sigma(\beta)$, and $t\in L^\times$ satisfies $N_{L|K}(t)=\alpha^{-n}$.
\end{customthm}

For the case when $S= \emptyset$ and $K$ is totally imaginary, we have already computed the cohomology ring in \cite{Ahlqvist--Carlson-cup}, as previously noted.  The following theorem is essential for determining the cohomology ring structure of $X= \Spec \OO_K$ when $K$ is not necessarily totally imaginary; in some sense, as in the situation when $S \neq \emptyset$, when $K$ has a real place, there is only one interesting class of cup products, and they are given by:

\begin{customthm}{\ref{thm:cup2}}
  Let $y$ and $z$ be elements in $H^1(X,\ZZ/n\ZZ)$. Suppose that $y$ is represented by a cyclic extension $L/K$, unramified outside of $S$, together with a choice of generator $\sigma \in \Gal(L/K)$, and assume that $L/K$ has degree $d|n$. Then under the identifications $H^1(X,\ZZ/n\ZZ)\cong (\Cl^+(K)/n\Cl^+(K))^\sim$ and $H^2(X,\ZZ/n\ZZ)\cong (Z^1/B^1)^\sim$ we have that $y\cupp z\in (Z^1/B^1)^\sim$ satisfies the formula
    \[
      \langle y\cupp z,(b,\bb)\rangle=\langle z,\bb^{n^2/2d}N_{L|K}(I)^{n/d}\rangle
    \]
  where $I$ is an element in $\Div(L)$ such that $\bb^{n/d}=\divis(t)I/\sigma(I)$, where $t\in L^\times$ satisfies $N_{L|K}(t)=b^{-1}$.
\end{customthm}

The formulas we have found for the cup product in étale cohomology has already been applied to problems in number theory a number of times. We mention a few:
\begin{enumerate}
  \item The inverse Galois problem \cite{CarlsonSchlankUnramified};
  \item Verifying the unramifed Fontaine--Mazur conjecture in some special cases \cite{MaireUnramified};
  \item Arithmetic Chern Simons theory \cite{Ahlqvist--Carlson-cup};
  \item Mod 2 arithmetic Dijkgraaf-Witten invariants for certain real quadratic number fields \cite{Hirano}.
\end{enumerate}
The results of this paper gives a more complete picture of the \'etale cohomology of arithmetic curves and we hope that there are a lot more applications in number theory to be found in a not too distant future.

\subsection*{Notation and conventions}
Throughout the paper we use the following notation:
\begin{itemize}
  \item $K$ is a number field,
  \item $X=\Spec \OO_K$, where $\OO_K$ is the ring of integers of $K$,
  \item $\Omega$ is the set of places of $K$,
  \item $\Omega_\infty$ is the set of archimedean places of $K$,
  \item $\Omega_\CC$ is the set of complex places of $K$,
  \item $\Omega_\RR$ is the set of real places of $K$,
  \item $S\subseteq \Omega$ is a finite non-empty set of places containing $\Omega_\RR$,
  \item $S_f=S\setminus \Omega_\RR$,
  \item $U=X\setminus S_f$,
  \item $j\colon \Spec K\to X$ is the canonical inclusion,
  \item $U_{\fl}$ is the big fppf site of $U$
  \item $\mmu_n$ is the fppf-sheaf of $n$th roots of unity.
   \end{itemize}
   
\subsection*{Acknowledgements}
The second author is grateful to Tomer Schlank for helpful discussions on some of the material in this paper.

\section{The cohomology of a punctured arithmetic curve}\label{sec:cohomology-groups}
The purpose of this section is to compute the \'etale cohomology groups $H^i(U,\ZZ/n\ZZ)$.  We compute these cohomology groups by using essentially the same techniques as in \cite{Ahlqvist--Carlson-cup}, where we computed the cohomology groups $H^i(X,\ZZ/n\ZZ)$  for  $X= \Spec \OO_K$ with $K$ totally imaginary and $\OO_K$ the ring of integers of $K.$ The strategy to compute the cohomology groups $H^i(U,\ZZ/n\ZZ)$ is to first compute the compactly supported \emph{flat} cohomology groups $H^i_c(\Ufl,\mmu_n)$ by resolving $\mmu_n$,  the fppf sheaf of $n$th roots of unity, and then to use Artin--Verdier duality, which gives that $H^i(U,\ZZ/n\ZZ)$ is the Pontryagin dual to $H^i_c(\Ufl,\mmu_n)$. We will now start by recalling the definition of compactly supported fppf cohomology, following \cite{MilneADT} and \cite{Demarche--Harari}. We then proceed by stating Artin--Verdier duality, and then we  compute the cohomology groups $H^i_c(\Ufl,\mmu_n)$, and thus, by Pontryagin duality, we also compute $H^i(U,\ZZ/n\ZZ)$.

Before we compute the cohomology groups we will establish some notation.  For an non-archimedean place $p$,  $K_p$ denotes the usual completion of $K$ and at an archimedean place, $K_p = \mathbb{R}$ if $p$ is real, and $K_p = \mathbb{C}$ if $p$ is complex. We define $\OO_p$ to be unit ball in $K_p$ if $p$ is non-archimedean, and we let $\OO_p = K_p$ if $p$ is archimedean. Let us now define $I_K$ as the group of (non-complex) idèles
  \[I_{K}=\prod_{p\in \Omega\setminus \Omega_\CC}'K^\times_p
  \,,\]
where $\prod'$ means that we take the restricted product with respect to the subgroup $\prod_{p \in \Omega \setminus \Omega_\CC} \OO_p^\times$.
Define
  \[\begin{split}
  & C_{K} = I_{K}/K^\times\,, \\
  & C_S(K)=I_K/K^\times U_{K,S}\,, \\
  \end{split}\]
where
  \[U_{K,S}=\prod_{p\in S}\{1\}\times \prod_{p\in U}\OO_p^\times
  \,\]
  and $K$ is embedded diagonally.

Let us also define $I_{K,S}=\prod_{p\in S}K_p^\times$, $\OO_{K,S}=\{x\in K:|x|_p\leq 1\mbox{ for all }p\notin S\}$, and $C_{K,S}=I_{K,S}/\OO_{K,S}^\times$. We have a  canonical inclusion $C_{K,S}\hookrightarrow C_S(K)$ induced by the map $I_{K,S}  \rightarrow I_K$ which takes $( \alpha_p)_{p \in S}$ to the idèle which is $\alpha_p$ at places in $S$ and $1$ outside of $S$. By \cite[Proposition 8.3.5]{Neukirch--Schmidt--Wingberg} we have an exact sequence of topological groups
  \[1\to C_{K,S}\to C_S(K)\to \Cl_S(K)\to 1\,,\] where $\Cl_S(K)$ is the $S$-ideal class group considered as a discrete group.

\subsection*{Flat cohomology with compact support}
Following \cite{Demarche--Harari} and \cite{Geisser--Schmidt}, we now proceed by defining flat cohomology with compact support.

Let $F$ be a bounded complex of abelian sheaves on $\Ufl$.
Define $Z'= \amalg_{p\in X\setminus U}\Spec K_p$ and write $\gamma\colon Z'\to U$ for the canonical morphism.
By \cite[\href{https://stacks.math.columbia.edu/tag/06VX}{Tag 06VX}]{stacks-project} the big fppf site has enough points. According to \cite[C.2.2.11]{JohnstoneElephant2} one can choose a \emph{set} of points such that the points are jointly conservative. We now
let $F\to G(F)$ be the Godement resolution with respect to a fixed choice of a set of points which are jointly conservative.  We thus resolve $F$ degreewise so that $G(F)$ is a double complex. Then we have the unit map
    \[\Gamma(U,G(F))\to \Gamma(U,\gamma_*\gamma^*G(F))=\Gamma(Z',\gamma^*G(F))\,.\]
For each $\nu\in \Omega_\RR$ we let $a^\nu\colon (\Sch/\Spec \RR)_{\fl}\to (\Spec\RR)_\et$ be the canonical morphism of sites. Then $a^\nu_*$ is exact and hence $a^\nu_*G(F)_\nu$ is a resolution of $a^\nu_*F_\nu$ into acyclics.
As explained in \cite[§2]{Geisser--Schmidt}, there is a resolution $D^\bullet(a^\nu_*F_\nu)$ of $a^\nu_*F_\nu$ which is pointwise acyclic, and splicing $D^\bullet(a^\nu_*F_\nu)$ and $a^\nu_*G(F)_\nu$ together, one gets a functorial complete (pointwise) acyclic resolution $\hat{G}(F_\nu)$ of $a^\nu_*F_\nu$. The resolution $\hat{G}(F_\nu)$ computes the Tate-hypercohomology of $a^\nu_*F_\nu$ and there is a canonical map $a^\nu_*G(F)_\nu\to \hat{G}(F_\nu)$.

We define
  \[
    R\hat{\Gamma}_c(\Ufl,F)=\Tot(C(\Gamma(U,G(F))\to \Gamma(Z',\gamma^*G(F))\oplus \bigoplus_{\nu\in \Omega_\RR}\Gamma(K_\nu,\hat{G}(F_\nu)))[-1])\,,
  \]
and 
\[
  R\Gamma_c(\Ufl,F)=\Tot(C(\Gamma(U,G(F))\to \Gamma(Z',\gamma^*G(F))\oplus \bigoplus_{\nu\in \Omega_\RR}\Gamma(K_\nu,a^\nu_*G(F)_\nu))[-1])\,,
\]
 where $C(-)$ denotes the mapping cone. Then we have a canonical natural transformation 
  \[
    R\Gamma_c(\Ufl,F) \to R\hat{\Gamma}_c(\Ufl,F)\,.  
  \] 
We will also at times view $R\Gamma_c(\Ufl,F)$ and $R\hat{\Gamma}_c(\Ufl,F)$ as objects of $D(\Ab)$. 
\begin{df}
  For a bounded complex $F$ of abelian sheaves on $\Ufl$, we define
      \[
          H^i_c(\Ufl,F)=H^i(R\hat{\Gamma}_c(\Ufl,F))\,.
      \]
\end{df}
The following two definitions will be useful in Section \ref{sec:cup}. Given a bounded complex $F$ of abelian sheaves on $\Ufl$, we now define, functorially in $F$, an object $\Gamma_c(\Ufl,F)$, which will be a complex of abelian groups.
\begin{df} \label{def:global}
For a bounded complex $F$ of abelian sheaves on $\Ufl$, we define $\Gamma_c(\Ufl,F)$ to be the complex given by
\[
  H^0(C(\Gamma(U,G(F))\to \Gamma(Z',\gamma^*G(F))\oplus \bigoplus_{p\in \Omega_\RR}\Gamma(K_\nu,a^\nu_*G(F)_\nu))[-1])
\]
where we take cohomology ``vertically''.
\end{df}

Hence we have canonical natural transformations 
  \[
    \Gamma_c(\Ufl,F)\to R\Gamma_c(\Ufl,F) \to R\hat{\Gamma}_c(\Ufl,F)\,. 
  \]

\subsection*{Cohomology of $U$}
We will now resolve $\mmu_n$. We have $\mmu_n\cong\HOM(\ZZ/n\ZZ,\GG_m)$, where we are taking the internal hom in the category of fppf sheaves. We define resolutions $\Ei\to \ZZ/n\ZZ$ and $\GG_{m,U}\to \C$ by
  \begin{equation}
    \Ei  = (\ZZ\xrightarrow{n}\ZZ)  \label{res1}
    \end{equation}
    and
    \begin{equation}
    \C  =( j_*\GmK\xrightarrow{\divis}\DIV_U )\label{res2}
  \end{equation}
where $\Ei$ is concentrated in degree $-1$, $0$, $\C$ is concentrated in degree 0, 1, and
  \[
    \DIV_U=\bigoplus_{p\in U}(i_p)_*\ZZ
  \]
is the sheaf of Weil divisors, where $i_p\colon \Spec \kappa(p) \rightarrow U$ is the inclusion of the residue field (for a proof that $\C$ is a resolution of $\GG_m$, see \cite[II,3.9]{MilneEtale}).
The complex $\HOM(\Ei,\C)$ is a resolution of $\mmu_n$ for the following reasons. First, since the map $\cdot n \colon \GG_m \rightarrow \GG_m$ is an epimorphism of sheaves, $\HOM(\Ei,\mathbb{G}_m)$ is a resolution of $\mmu_n$. Then, since $\Ei$ is a complex of locally free sheaves, $\HOM(\Ei,-) = R\HOM(\Ei,-)$, so $\HOM(\Ei,-)$ preserves the quasi-isomorphism $\GG_m \rightarrow \C$, which gives that $\HOM(\Ei,\C)$ is a resolution of $\mmu_n$.

There is a horizontal filtration on $R\hat{\Gamma}_c(\Ufl,\HOM(\Ei,\C))$ coming from the columns of the corresponding double complex
\[
  L=C(\Gamma(U,G(\HOM(\Ei,\C)))\to \Gamma(Z',\gamma^*G(\HOM(\Ei,\C)))\oplus \bigoplus_{p\in \Omega_\RR}\Gamma(K_\nu,\hat{G}(\HOM(\Ei,\C)_\nu)))[-1]\,,
\]
so one obtains a spectral sequence
    \[
      E_2^{s,t}=H^sH^t(\Gamma(L)) \Rightarrow H^{s+t}(R\hat{\Gamma}_c(\Ufl,\HOM(\Ei,\C)))\cong H^{s+t}_c(\Ufl,\mmu_n)
    \]
with $E_1^{s,t}=H^t(L^{s,\bullet})$. We will now use this spectral sequence to calculate $H^i_c(\Ufl,\mmu_n)$, and proceed by calculating the $E_2$-page. The following well-known lemma is crucial for computing the $E_2$-page.

\begin{lem}
For any $\nu\in \Omega_\RR$, we have
  \[
    H_T^i(\Gal(\overline{K}_\nu/K_\nu),\GG_{m,K_\nu})\cong
    \begin{cases}
        K^\times_\nu/2K^\times_\nu & \mbox{if }i\mbox{ is even}\,, \\
        0 & \mbox{if }i\mbox{ is odd}\,.
    \end{cases}
  \]
\end{lem}

\begin{proof}
  Since $H_T^i(\Gal(\CC/\RR),\GG_{m,\RR})$ is periodic in $i$ with period 2 \cite[Theorem I.6.1]{Neukirch-Bonn}, it is enough to compute for instance $H_T^0(\Gal(\CC/\RR),\GG_{m,\RR})$ and $H_T^1(\Gal(\CC/\RR),\GG_{m,\RR})$. This is left to the reader.
\end{proof}

For $0\leq t\leq 2$ we have that $E_1^{s,t}=H^t(L^{s,\bullet})$ is
  \[
      \begin{matrix}
        \Br(K) & \xrightarrow{d^{0,2}} & \Br(K)\oplus \prod_{p\nmid\infty} \QQ/\ZZ \oplus \prod_{\nu\in \Omega_\RR}K^\times_\nu/2K^\times_\nu & \xrightarrow{d^{1,2}} & \prod_{p\nmid\infty} \QQ/\ZZ \oplus \prod_{\nu\in \Omega_\RR}K^\times_\nu/2K^\times_\nu \\
        0 & \to & 0 & \to & 0 \\
        K^\times & \xrightarrow{d^{0,0}} & K^\times\oplus \Div U\oplus\prod_{p\in S_f} K_p^\times \oplus \prod_{\nu\in \Omega_\RR}K^\times_\nu/2K^\times_\nu & \xrightarrow{d^{1,0}} & \Div U\oplus \prod_{p\in S_f} K_p^\times \oplus \prod_{\nu\in \Omega_\RR}K^\times_\nu/2K^\times_\nu
      \end{matrix}
  \]
where $d^{0,2}$ is injective since it is the invariant map into the second factor. Hence the $E_2$-page may be pictured as
  \[
    \begin{sseq}{0...4}{-2...2}
    \ssdropbull
    \ssmoveto{2}{-2}
    \ssdropbull
    \ssmoveto{1}{-2}
    \ssdropbull
    \ssmoveto{1}{0}
    \ssdropbull
    \ssmoveto{2}{0}
    \ssdropbull
    \ssmoveto{1}{2}
    \ssdropbull
    \ssmoveto{2}{2}
    \ssdropbull
    \end{sseq}
  \]
in vertical degrees $-2\leq t\leq 2$. We see that $E_2^{1,-2}$, $E_2^{1,2}$, and $E_2^{2,2}$ will not contribute to $H^{i}_c(\Ufl,\mmu_n)$ for $1 \leq i\leq 2$. Hence we conclude that for $1 \leq i\leq 2$, $H^{i}_c(\Ufl,\mmu_n)$ is the $i$th cohomology of the complex
  \[
    K^\times \xrightarrow{d^{0,0}} K^\times\oplus \Div U\oplus\prod_{p\in S_f} K_p^\times \oplus \prod_{\nu\in \Omega_\RR}K^\times_\nu/2K^\times_\nu \xrightarrow{d^{1,0}} \Div U\oplus \prod_{p\in S_f} K_p^\times \oplus \prod_{\nu\in \Omega_\RR}K^\times_\nu/2K^\times_\nu
  \]
where
  \[
      d^{0,0} =
      \begin{pmatrix}
        -n \\
        \divis \\
        \eta \\
        (\iota_\nu)_\nu
      \end{pmatrix}\mbox{ and }
      d^{1,0} =
      \begin{pmatrix}
      \divis & n & 0 & 0 \\
      \eta & 0 & n & 0 \\
      (\iota_\nu)_\nu & 0 & 0 & n
      \end{pmatrix}\,.
  \]

\begin{lem}\label{lem:dual-groups}
Let $K_+$ be the set of totally positive elements in $K$ and $\Cl^+(K)$ the narrow class group of $K$. Furthermore, let $r$ be the number of real places of $K$. For $S=\Omega_\RR$ we have
  \[
    H^i_c(X_\fl,\mmu_n)=
    \begin{cases}
      (\RR^\times/n\RR^\times)^r & i<0\, \\
      \mu_n(K_+) \oplus (\RR^\times/n\RR^\times)^r & i=0\,, \\
      Z^1/B^1 & i=1\,, \\
      \Cl^+(K)/n\Cl^+(K) & i=2\,,\\
      \ZZ/n\ZZ & i=3\,,\\
      0 & i\geq 4\,,
    \end{cases}
  \]
where
  \[
      \begin{split}
        Z^1 & = \{(a,I)\in K_+^\times\oplus \Div U : \divis(a)I^n=1\}\,, \\
        B^1 & = \{(b^{-n},\divis(b)):b\in K^\times_+\}\,.
      \end{split}
  \]
For $S\neq \Omega_\RR$ we have
  \[
  H^i_c(\Ufl,\mmu_n)=\begin{cases}
    (\RR^\times/n\RR^\times)^r & i \leq 0\, \\
    C_S(K)[n] & i=1\,, \\
    C_S(K)/nC_S(K) & i=2\,,\\
    \ZZ/n\ZZ & i=3\,,\\
    0 & i\geq 4\,.
  \end{cases}
  \]
\end{lem}


Note in the above that $K^+=K$ when $K$ is totally imaginary. By flat Artin--Verdier duality (see \cite[Thm 1.1]{Demarche--Harari})  we get:

\begin{thm}\label{thm:coh-groups}
Let $K_+$ be the set of totally positive elements in $K$ and $\Cl^+(K)$ the narrow class group of $K$. Furthermore, let $r$ be the number of real places of $K$ and for any abelian group $A$ we denote by $A^\sim=\Hom(A,\QQ/\ZZ)$, the Pontryagin dual of $A$. For $S=\Omega_\RR$ we have
  \[
    H^i(X,\ZZ/n\ZZ)\cong
    \begin{cases}
      \ZZ/n\ZZ & i=0\,, \\
      (\Cl^+(K)/n\Cl^+(K))^\sim & i=1\,, \\
      (Z^1/B^1)^\sim & i=2\,,\\
      (\mu_n(K_+) \oplus (\RR^\times/n\RR^\times)^r)^\sim & i=3\,,\\
      ((\RR^\times/n\RR^\times)^r)^\sim & i\geq 4 \ .
    \end{cases}
  \]
where
  \[
      \begin{split}
        Z^1 & = \{(a,I)\in K_+^\times\oplus \Div U : \divis(a)I^n=1\}\,, \\
        B^1 & = \{(b^{-n},\divis(b)):b\in K^\times_+\}\,.
      \end{split}
  \]
For $S\neq \Omega_\RR$ we have
  \[
  H^i(U,\ZZ/n\ZZ)=\begin{cases}
    \ZZ/n\ZZ & i=0\,, \\
    (C_S(K)/nC_S(K))^\sim & i=1\,, \\
    (C_S(K)[n])^\sim & i=2\,,\\
    ((\RR^\times/n\RR^\times)^r)^\sim & i\geq 3\,.
  \end{cases}
  \]
\end{thm}

\begin{proof}[Proof of Lemma \ref{lem:dual-groups}]
We saw that, for $1 \leq i\leq 2$, $H^i_c(\Ufl,\mmu_n)$ is the $i$th cohomology of the complex
\[
  K^\times \xrightarrow{d^{0}} K^\times\oplus \Div U\oplus\prod_{p\in S_f} K_p^\times \oplus \prod_{\nu\in \Omega_\RR}K^\times_\nu/2K^\times_\nu \xrightarrow{d^{1}} \Div U\oplus \prod_{p\in S_f} K_p^\times \oplus \prod_{\nu\in \Omega_\RR}K^\times_\nu/2K^\times_\nu
\]
where
\[
    d^{0} =
    \begin{pmatrix}
      -n \\
      \divis \\
      \eta \\
      (\iota_\nu)_\nu
    \end{pmatrix}\mbox{ and }
    d^{1} =
    \begin{pmatrix}
    \divis & n & 0 & 0 \\
    \eta & 0 & n & 0 \\
    (\iota_\nu)_\nu & 0 & 0 & n
    \end{pmatrix}\,.
\]
We first consider the case $S=\Omega_\RR$.  The cases $i=1$ and $i=2$ follows by the approximation theorem \cite[§3, (3.4)]{Neukirch-ANT}, which implies that the map $K^\times \to \prod_{\nu\in \Omega_\RR}K^\times_\nu/2K^\times_\nu$ is surjective. In the negative degrees, we just get the Tate cohomology of $\mmu_n$ at the real places.

Next we consider the case $S\neq \Omega_\RR$.
The map $\eta$ is injective and hence $H^0_c(\Ufl,\mmu_n)=0$.

Note that if $(b,\bb, \alpha,\beta)\in \ker d^1$, then $\eta(b)=\alpha^{-n}$ and since $\eta$ is injective, we get that $b$ is determined by $\alpha$ and in particular, if $\alpha=1$ then $b=1$. This means that the projection
  \[K^\times\oplus \Div U\oplus\prod_{p\in S_f} K_p^\times \oplus \prod_{\nu\in \Omega_\RR}K^\times_\nu/2K^\times_\nu\to \Div U\oplus\prod_{p\in S_f} K_p^\times \oplus \prod_{\nu\in \Omega_\RR}K^\times_\nu/2K^\times_\nu\]
gives an isomorphism
  \[
    \ker d^1/\im d^0\cong ((\Div U\oplus \prod_{p\in S_f}K_p^\times\oplus \prod_{\nu\in \Omega_\RR}K^\times_\nu/2K^\times_\nu)/K^\times)[n]\,.
  \]
We have a canonical surjection
  \[
    C_S(K) \to (\Div U\oplus \prod_{p\in S_f}K_p^\times\oplus \prod_{\nu\in \Omega_\RR}K^\times_\nu/2K^\times_\nu)/K^\times\,.
  \]
  with kernel $(1,1, \prod_{\nu\in \Omega_\RR}(K^\times_\nu)^2)K^\times/K^\times$. The kernel is thus isomorphic to $\prod_{\nu\in \Omega_\RR}(K^\times_\nu)^2$ since $K^\times \rightarrow K_p^\times$ is injective for a non-archimedean prime.
  But $\prod_{\nu\in \Omega_\RR}(K^\times_\nu)^2$ is uniquely divisible, so upon applying the $n$-torsion functor, we get an isomorphism
  \[
    C_S(K)[n]  \cong ((\Div U\oplus \prod_{p\in S_f}K_p^\times\oplus \prod_{\nu\in \Omega_\RR}K^\times_\nu/2K^\times_\nu)/K^\times)[n]\,.
  \]
  Thus our claim for $H^1_c(\Ufl,\mmu_n)$ follows, as does the claim for $H^2_c(\Ufl,\mmu_n)$. To see that $H^0_c(\Ufl,\mmu_n) = \mu_n(K_+) \oplus (\RR^\times/n\RR^\times)^r$ note first that if $K$ has a real place, then the right hand side is isomorphic to $(\RR^\times/n\RR^\times)^r$ since $\mu_n(K_+)$ vanishes. If on the other hand $K$ is totally imaginary, the right hand side is isomorphic to $\mu_n(K)$. 
  By using the long exact sequence
  \[
    \dots \to H^i_c(\Ufl,\GG_m)\to H^i(\Ufl,\GG_m)\to \bigoplus_{p\in S}H^i(K_p,\GG_m )\to \dots
  \]
  together with the fact that $H^{-1}(K_p,\GG_m) =0$ for $p \in S$ (where we interpret $H^{-1}(K_p,\GG_m)$ as Tate cohomology in the archimedean situation), we get that $H^0_c(\Ufl,\GG_m) = \OO_{U,+}^\times$, i.e., the totally positive units of $U$. By considering the long exact sequence in compactly supported cohomology associated to the Kummer sequence
  \[
    0\to \mmu_n\to \GG_m\to \GG_m\to 0
  \]
  we then get that $H^0_c(\Ufl,\mmu_n) $ sits in a short exact sequence
  \[
    0 \to  (\RR^\times/n\RR^\times)^r \to H^1_c(\Ufl,\GG_m) \to \mu_n(K_+) \to 0\,.
  \]
  By our previous remarks on when the leftmost term and the rightmost term are zero, our claim follows. To find the value of $H^i_c(\Ufl,\mmu_n)$ for $i = 3$, one proceeds with the Kummer sequence once again and uses that $H^2_c(\Ufl,\mmu_n)=0 $ (see \cite[II,Proposition 2.6]{MilneADT}) is divisible and that $H^3_c(\Ufl,\GG_m) = \QQ/\ZZ$. To lastly see that $H^i_c(\Ufl,\mmu_n)=0$ for $i>3$ one can use flat Artin--Verdier duality to see that they then coincide with the Pontryagin dual of $H^{3-i}(\Ufl,\ZZ/n\ZZ),$ and these vanish since cohomology groups in negative degrees always are zero.
%
\end{proof}

\begin{remark}
  When $K$ is totally imaginary, a slightly more geometric approach to computing the groups $H^{1}_c(\Ufl,\mmu_n)$ and $H^{2}_c(\Ufl,\mmu_n)$ without using resolutions is as follows:

  Let us first consider the case $S=\emptyset$. The group $H^1(X_\fl,\mmu_n)$ classifies $\mmu_n$-torsors up to isomorphism. Let $B\mmu_n(X)$ be the groupoid of $\mmu_n$-torsors. Any $\mmu_n$-torsor is isomorphic to one on the form $\Spec \mathcal{A}\to X$ where $\A$ is the $\OO_X$-algebra with module structure
    \[
      \A\cong \bigoplus_{i=0}^{n-1} \Li^{\otimes -i}
    \]
  for $\Li$ a line bundle on $X$ and with multiplication determined by an isomorphism $\Li^{\otimes-n}\to \OO_X$ with we think of as an invertible global section $u\in \Gamma(X,\Li^{\otimes n})^\times$. Hence we see that $B\mmu_n(X)$ is equivalent to the groupoid of pairs $(\Li,u)$ and with morphisms $(\Li',u')\to (\Li,u)$ given by isomorphisms $\varphi\colon \Li'\to \Li$ such that $\varphi^{\otimes n}u'=u$.
  A pair $(\Li,u)$ is \emph{trivial} if there is an isomorphism $\varphi\colon\Li\cong \OO_X$ such that $\varphi^{\otimes n}u=1$, i.e., $u$ is an $n$th power.
  Writing a line bundle $\Li\cong \OO_X(I)$ for some fractional ideal $I$ we may interpret $u\in \Gamma(X,\Li^{\otimes n })^\times\subset K^\times$ as a generator of the ideal $I^n$. This gives
    \[
      \pi_0B\mmu_n(X) \cong Z_1/B_1\cong H^1(X_\fl, \mmu_n)
    \]
  with
    \[\begin{split}
      Z_1 & = \{(\bb,b)\in \Div(X)\oplus K^\times:\divis(b)\bb^n=1\}\,, \\
      B_1 & = \{(\divis(b),b^{-n}):b\in K^\times\}\,.
    \end{split}\]

  Now to the case $S\neq \emptyset$: Write $i_p\colon \Spec \OO_p\to U$ for the canonical map where $\OO_p$ is the complete local ring at the prime $p$. The group $H^{1}_c(\Ufl,\mmu_n)$ classifies objects in the groupoid $B_c\mmu_n(U)$ up to isomorphism. That is, the groupoid with objects that are triples $(\Li, u, (\alpha_p)_{p\in S})$ where
  \begin{enumerate}
    \item $\Li$ is a line bundle on $U$,
    \item $u\in \Gamma(U,\Li^{\otimes n})^\times$, and
    \item $\alpha_p\colon i_p^*\Li\to \OO_p$ is an isomorphism
  \end{enumerate}
  such that $\alpha_p^ni_p^*u=1$. A morphism $(\Li', u', (\alpha'_p)_{p\in S})\to(\Li, u, (\alpha_p)_{p\in S})$ is given by an isomorphism $\varphi\colon \Li'\to \Li$ such that $\varphi^{\otimes n}u'=u$ and $\alpha_p i_p^*\varphi=\alpha'_p$.
  An object $(\Li, u, (\alpha_p)_{p\in S})$ is \emph{trivial} if there is an isomorphism $\varphi\colon \OO_X\to \Li$ such that $u\varphi^n=1$ and $(i_p^*u)\alpha_p^n=1$. As before, by interpreting $\Li$ via fractional ideals we may think of $\alpha_p$ as a choice of generator in $K_p^\times$ of the invertible module $i_p^*\Li$. We conclude that
    \[
      \pi_0B_c\mmu_n(U) \cong Z_{S}^1/B_{S}^1\cong H^1_c(\Ufl,\mmu_n)
    \]
  where
  \[\begin{split}
    Z_{S}^1 & = \{(\bb,b,a)\in \Div(X)\oplus K^\times\oplus \prod_{p\in S}K_p^\times:\divis(b)\bb^n=1\mbox{ and }\eta(b)a^n=1\}\,, \\
    B_{S}^1 & = \{(\divis(b),b^{-n},\eta(b)):b\in K^\times\}\,,
  \end{split}\]
  as in the proof of Lemma \ref{lem:dual-groups}.

  To compute $H^2_c(\Ufl,\mmu_n)$ we use the Kummer sequence
    \[
      0\to \mmu_n\to \GG_m\to \GG_m\to 0
    \]
  which is exact since we are working in the fppf topology. Hence we obtain an exact sequence
    \[
      H^1_c(\Ufl,\GG_m)\xrightarrow{\delta} H^2_c(\Ufl,\mmu_n)\to H^2_c(\Ufl,\GG_m)=0
    \]
  where the last equality follows from the exact sequence
    \[
      0=\bigoplus_{p\in S} H^1(\OO_p,\GG_m)\to H^2_c(\Ufl,\GG_m)\to H^2(X_\fl,\GG_m)=0
    \]
  of \cite[Proposition III.0.4.(c), Remark III.0.6]{MilneADT}.

  Similarly as before, $H^2_c(\Ufl,\mmu_n)$ classifies $\mmu_n$-gerbes with certain trivializations over the complete local rings at the primes in $S$ up to equivalence. We translate $\delta$ to the language of groupoids as follows (see e.g. \cite[12.2.5]{Olsson-Book}): With notation as before, $H^1_c(\Ufl,\GG_m)$ corresponds to the groupoid $B_c\GG_m(U)$ of pairs $(\Li, (\alpha_p)_{p\in S})$ where $\Li$ is a line bundle on $U$ and the $\alpha_p$'s are trivializations as before over the complete local rings of the primes in $S$.
  The map $\delta$ sends such a pair $(\Li, (\alpha_p)_{p\in S})$ to the pair $(U(\sqrt[n]{\Li}),\delta(\alpha_p)_{p\in S})$, where $U(\sqrt[n]{\Li})$ is the $\mmu_n$-gerbe classifying $n$th roots of the line bundle $\Li$ and
    \[
      \delta(\alpha_p)\colon \Spec\OO_p\times_U U(\sqrt[n]{\Li})\simeq \Spec\OO_p(\sqrt[n]{\OO_p})=B\mmu_n
    \]
  is the trivialization induced by $\alpha_p\colon i_p^*\Li\cong \OO_p$. We have an essentially surjective map
    \[
      \Div(U)\oplus \prod_{p\in S}K_p^\times \to \delta(B_c\GG_m)
    \]
  sending a pair $(I,(a_p)_{p\in S})$ to $(U(\sqrt[n]{\OO_U(I)}),(\tilde{a}_p)_{p\in S})$ there $\tilde{a}_p$ is the trivialization obtained by the choice of generator $a_p$ for $I$ over $\OO_p$. The pair $(U(\sqrt[n]{\OO_U(I)}),(\tilde{a}_p)_{p\in S})$ will be \emph{trivial} in the appropriate sense exactly when $I$ is in the image of the map
    \[
    \begin{pmatrix}
      \divis & n & 0 \\
      \eta & 0 & n
    \end{pmatrix}\colon K^\times\oplus \Div U\oplus \prod_{p\in S}K_p^\times\to
    \Div U\oplus \prod_{p\in S}K_p^\times\,.
    \]
\end{remark}

\section{The cup product} \label{sec:cup}
  In this section we compute the cohomology ring of an arithmetic curve. We start by giving some heuristics for how we will compute the cup product map. Assume first that $S_f\neq \emptyset$ and that $K$ has no real places. By analogy we may think of $U$ as the complement of a knot embedded in a 3-manifold. In particular, (here we use the assumption that $K$ has no real places), the étale cohomological dimension is 2. The cup product for $H^*(U,\ZZ/n\ZZ)$ is graded commutative and by considering the cohomological dimension, the only non-trivial task is to compute the map
    \[c_y:=y\cupp-\colon H^1(U,\ZZ/n\ZZ)\to H^{2}(U,\ZZ/n\ZZ)\, \\\]
  for any $y\in H^1(U,\ZZ/n\ZZ)$. We will compute the map $c_y$ as follows: we have a canonical isomorphism
    \[H^1(U,\ZZ/n\ZZ)\cong \Ext^1_U(\ZZ/n\ZZ,\ZZ/n\ZZ)\]
  and the group $\Ext^1_U(\ZZ/n\ZZ,\ZZ/n\ZZ)$ classifies extensions
    \[
      0\to \ZZ/n\ZZ\to E\to \ZZ/n\ZZ\to 0\,,
    \]
  of abelian étale sheaves on $U$.
  By choosing an extension $E$ representing the class $y\in H^1(U,\ZZ/n\ZZ)$ we get that $y\cupp-=\delta_E$ where
    \[
      \delta_E\colon H^i(U,\ZZ/n\ZZ)\to H^{i+1}(U,\ZZ/n\ZZ)
    \]
  is the connecting homomorphism coming from the extension $E$. We will then compute the map which is dual to $\delta_E$ under Artin--Verdier duality by taking resolutions of the sheaves occurring in the dual exact sequence, and then, dualize once again, to get a formula for $\delta_E$.

\subsection*{S-torsors}
There is a simple description of $\ZZ/n\ZZ$-torsors $Y\to U$. Recall that if $d|n$ and $Z\to U$ is a $\ZZ/d\ZZ$-torsor, then we may form the induced $\ZZ/n\ZZ$-torsor $\Ind_{\ZZ/d\ZZ}^{\ZZ/n\ZZ}(Z)$ as in \cite[p.~11]{Ahlqvist--Carlson-cup}.
Every $\ZZ/n\ZZ$-torsor is then of the form $Y=\Spec R$, and if $Y$ is connected, then $L=\Frac(R)$ is a degree $n$ extension of $K$ which is unramified outside of $S_f$, and the canonical morphism $Y\to \Spec \OO_L\times_XU$ is an isomorphism. If $Y$ is not connected, then $Y$ is induced from a connected degree $d$ torsor $Z \to U$ where $d | n$. We summarize the above discussion in the following lemma, and the proof of the lemma is entirely analogous to the proof of \cite[Lemma 2.20]{Ahlqvist--Carlson-cup}.

\begin{lem}
  Let $Y\to U$ be a $\ZZ/n\ZZ$-torsor. Then there exists a degree $d|n$ extension $L/K$ which is unramified outside of $S_f$ and an isomorphism
    \[Y\cong \Ind_{\ZZ/d\ZZ}^{\ZZ/n\ZZ}(\Spec\OO_L \times_X U)\,.\] Further, $L/K$ is unique up to isomorphism.
\end{lem}

\subsection*{The extension associated to a torsor}
  The method is the same as in \cite{Ahlqvist--Carlson-cup}. For $y\in H^1(U,\ZZ/n\ZZ)$, choose a $\ZZ/n\ZZ$-torsor $\pi\colon Y\to U$ representing $y$. Then $Y$ is the restriction to $U$ of a ramified $\ZZ/n\ZZ$-torsor $Y'\to X$, where $Y'$ is of the form
    \[Y'=\Ind_{\ZZ/d\ZZ}^{\ZZ/n\ZZ}(Z)\]
  where $Z$ is the spectrum of the ring of integers of a degree $d|n$ extension Galois $L/K$ which is unramified outside $S_f$.
  Since $\pi$ is finite \'etale we get that $\pi_*=\pi_!$ and hence $\pi_*$ is left adjoint to $\pi^*$.
  The counit $N\colon \pi_*\pi^*\ZZ/n\ZZ\to \ZZ/n\ZZ$ is called the \emph{norm}. In algebraic geometry literature $N$ is usually called the \emph{trace}, but we use the name norm since it is standard number theoretic nomenclature.

\begin{remark}
  To clarify, we have two adjunctions
    \[\pi^*:\Sh(U_\et)\leftrightarrows \Sh(Y_\et):\pi_*\quad\mbox{ and }\quad \pi_*:\Sh(Y_\et)\leftrightarrows \Sh(U_\et):\pi^*\]
  where the functor on the left is \emph{left} adjoint to the functor on the right and vice versa.
\end{remark}

  We have a short exact sequence of abelian étale sheaves
    \begin{equation}\label{eq:tr-seq}0\to \ker N\to \pi_*\pi^*\ZZ/n\ZZ\xrightarrow{N} \ZZ/n\ZZ\to 0\,\end{equation}
  which we refer to as the \emph{norm sequence}.
  There is an equivalence of abelian categories between the category of locally constant sheaves split by $\pi$ and the category of $C_n$-modules, where $C_n$ is the cyclic group with $n$ elements and a fixed generator $e$. Under this equivalence the norm sequence corresponds to the short exact sequence of $C_n$-modules
      \begin{equation}\label{eq:tr-seq-mod}0\to \ker \epsilon\to \ZZ/n\ZZ[C_n]\xrightarrow{\epsilon} \ZZ/n\ZZ\to 0\,,\end{equation}
  where $C_n$ acts trivially on $\ZZ/n\ZZ$ and acts on $\ZZ/n\ZZ[C_n]$ by translation on the generators. The morphism $\epsilon\colon \ZZ/n\ZZ[C_n]\to \ZZ/n\ZZ$ sends $g\in C_n$ to 1. The $C_n$-module $\ker \epsilon$ is free as an $\ZZ/n\ZZ$-module on $\{g-1\}_{g\in C_n}$. We define a map
  $s\colon \ker \epsilon\to \ZZ/n\ZZ$ by sending $e^l-1$ to $l$.
  If we take the pushout of the sequence (\ref{eq:tr-seq-mod}) along $s$, we get a diagrams
    \[\begin{tikzcd}
      0\ar{r} & \ker \epsilon\ar{r}\ar{d}{s} &  \ZZ/n\ZZ[C_n]\ar{r}\ar{d} & \ZZ/n\ZZ\ar{r}\ar[equal]{d} & 0 \\
      0\ar{r} & \ZZ/n\ZZ\ar{r} &  P\ar{r} & \ZZ/n\ZZ\ar{r} & 0\,,
    \end{tikzcd}\]
  where $P=\ZZ[G]/(\ker \epsilon)^2$, and the short exact sequence of $C_n$-modules on the bottom line corresponds to a short exact sequence of locally constant $\pi$-split sheaves
    \begin{equation}\label{eq:trans-seq}0\to \ZZ/n\ZZ\to E\to \ZZ/n\ZZ\to 0\,\end{equation}
  which is called the \emph{transfer sequence} associated to the $\ZZ/n\ZZ$-torsor $Y\to U$.

\begin{lem}
  The connecting homomorphism $H^i(U,\ZZ/n\ZZ)\to H^{i+1}(U,\ZZ/n\ZZ)$ associated to the transfer sequence (\ref{eq:trans-seq}) is given by cup product with $y\in H^1(U,\ZZ/n\ZZ)$.
\end{lem}

\begin{proof}
  The proof of \cite[Lemma 3.1]{Ahlqvist--Carlson-cup} goes through with $\tilde{X}_{\et}$ replaced by $U_{\et}$.
\end{proof}

  Denote by $\delta_E$ the connecting homomorphism corresponding to (\ref{eq:trans-seq}) and $\delta_y$ the connecting homomorphism of (\ref{eq:tr-seq}). Then $\delta_E=f_*\circ \delta_y$ where $f_*\colon H^*(U,\ker N)\to H^*(U,\ZZ/n\ZZ)$ is the morphism on cohomology induced by $f$.

\subsection*{Duality and functoriality}

Let $A$ be a finite flat group scheme and denote by $D(A)=\HOM(A,\GG_m)$ its Cartier dual. As a consequence of \cite[Lemma 4.1]{Demarche--Harari}, we have a canonical pairing in the derived category of abelian groups
  \[
    R\hat{\Gamma}_c(\Ufl,A)\otimes^{\LL}R\Gamma(\Ufl,D(A))\to R\hat{\Gamma}_c(\Ufl,\GG_m)
  \]
which is functorial in $A$. But $R\hat{\Gamma}_c(\Ufl,\GG_m)$ is isomorphic via the trace map to $\QQ/\ZZ$ concentrated in degree 3. This gives an isomorphism $$R \hat{\Gamma}_c(\Ufl,A) \cong (R \Gamma(\Ufl,D(A)))^\sim[3]$$  (see \cite[Theorem 1.1]{Demarche--Harari}) in $D(\Ab).$
We then have an induced pairing in the category of abelian groups
  \[
    H_c^{3-i}(\Ufl,A)\times H^i(\Ufl,D(A))\to H^3_c(\Ufl,\GG_m)\cong \QQ/\ZZ
  \]
which again by \cite[Coroloary III.3.2]{MilneADT} and \cite[Theorem 1.1]{Demarche--Harari} is perfect and induces a duality between the group $H_c^{3-i}(\Ufl,A)$ and the group $H^i(\Ufl,D(A))$.

A morphism $\varphi\colon F\to G$ of finite flat group schemes induces a dual morphism $D(\varphi) \colon D(G)\to D(F)$ and hence we obtain maps
    \[\begin{split}
      \varphi_* & \colon R\Gamma(\Ufl,F) \to R\Gamma(\Ufl,G)\,, \\
      D(\varphi)_* & \colon R\hat{\Gamma}_c(\Ufl,D(G)) \to R\hat{\Gamma}_c(\Ufl,D(F))\,,
    \end{split}\]
and by functoriality of  the above pairing, the diagram
  \begin{equation}\label{eq:comm-dual}\begin{tikzcd}[column sep=40pt]
    R\Gamma(\Ufl,F)\ar{r}{\varphi_*}\ar{d}{\cong} & R\Gamma(\Ufl,G)\ar{d}{\cong} \\
    R\hat{\Gamma}_c(\Ufl,D(F))[3]^\sim\ar{r}{D(\varphi)_*[3]^\sim} & R\hat{\Gamma}_c(\Ufl,D(G))[3]^\sim
  \end{tikzcd}\end{equation}
commutes, where $\sim$ denotes the functor $\RHom(-,\QQ/\ZZ)$.
%
%

We will also need the following lemma which is a direct consequence of \cite[Lemma 4.3]{Demarche--Harari}:

\begin{lem}
  Let $0\to A\to B\to C\to 0$ be a short exact sequence of finite flat group schemes and let $0\to D(C)\to D(B)\to D(A)\to 0$ be the short exact sequence obtained by Cartier duality. Then, for all $i,j\geq 0$, we have connecting homomorphisms $\delta_i\colon H^{i}(U,C)\to H^{i+1}(U,A)$ and $\delta'_j\colon H_c^{j}(\Ufl,D(A))\to H_c^{j+1}(\Ufl,D(C))$, and the diagram
  \[
    \begin{tikzcd}H_c^{j}(\Ufl,D(A))\times H^{i+1}(U,A)\ar{r}{\langle-,-\rangle}\ar[shift right=15]{d}{\delta'_j} & H^{i+j+1}_c(\Ufl,\GG_{m,U})\ar[equal]{d} \\
    H_c^{i+1}(\Ufl,D(C))\times H^{i}(U,C)\ar{r}{\langle-,-\rangle}\ar[shift right=15]{u}{\delta_i} & H^{i+j+1}_c(\Ufl,\GG_{m,U})
    \end{tikzcd}
  \]
  commutes, in the sense that for all $a\in H^{i}(U,C)$ and $b\in H_c^{j}(\Ufl,D(A))$ we have
  \[
    \langle b,\delta_i(a)\rangle=\langle\delta'_j(b),a\rangle\,.
  \]
\end{lem}

As a consequence we get that
  \[
    \delta_{i}=\Hom(\delta'_{3-i},\QQ/\ZZ)
  \]
when identifying $H^i(U,\ZZ/n\ZZ)\cong \Hom(H_c^{3-i}(\Ufl,\mmu_n),\QQ/\ZZ)$\,.

\subsection*{Computing the cup product}
  For $y\in H^1(U,\ZZ/n\ZZ)$ we will compute the dual
    \[c_y^\sim\colon H^1_c(\Ufl,\mmu_n)\to H^2_c(\Ufl,\mmu_n)\]
  of $c_y\colon H^1(U,\ZZ/n\ZZ)\to H^2(U,\ZZ/n\ZZ)$. We saw that $c_y=f_*\circ \delta_y$ and from the argument above we get the following description of the dual $c_y^\sim$:

\begin{lem}
  The map $c_y^\sim\colon R\hat{\Gamma}_c(\Ufl,\mmu_n)\to R\hat{\Gamma}_c(\Ufl,\mmu_n)[1]$ is given by $c_y^\sim=\delta_y^\sim\circ D(f)_*$, where
   \[\delta_y^\sim=\RHom(\delta_y,\QQ/\ZZ[-3])\]
  and $\delta_y\colon \ZZ/n\ZZ\to \ker N[1]$ is the connecting homomorphism.
\end{lem}

  Consider the diagram
    \begin{equation}\label{eq:diag-for-cup}\begin{tikzcd}
      0\ar{r} & \ker N\ar{r}{u}\ar{d}{f} & \pi_*\pi^*\ZZ/n\ZZ\ar{r}{N} & \ZZ/n\ZZ\ar{r} & 0 \\
      & \ZZ/n\ZZ & & & \,.
    \end{tikzcd}\end{equation}
  In addition to the resolution $\Ei\to \ZZ/n\ZZ$, where $\Ei$ is as in equality (\ref{res1}), and $\pi_*\pi^*\Ei\to \pi_*\pi^*\ZZ/n\ZZ,$ we define a resolution $\K\to \ker N$ by
    \[\K = \quad \ZZ\xrightarrow{\begin{pmatrix}
      n \\ -\Delta
    \end{pmatrix}} \ZZ\oplus \pi_*\pi^*\ZZ\xrightarrow{\begin{pmatrix}
      \Delta & n
    \end{pmatrix}} \pi_*\pi^*\ZZ\]
  where $\Delta\colon \ZZ\to \pi_*\pi^*\ZZ$ is the unit for the adjunction
    \[\pi^*:\Sh(U_\et)\leftrightarrows \Sh(Y_\et):\pi_*\,.\]
  To see that this really gives a resolution one may once again consider the sheaves as $C_n$-modules.

  The diagram (\ref{eq:diag-for-cup}) is then quasi-isomorphic to a diagram
    \begin{equation}\label{eq:diag-for-cup-2}\begin{tikzcd}
      0\ar{r} & \K\ar{r}{\hat{u}}\ar{d}{\hat{f}} & \pi_*\pi^*\Ei\ar{r}{\hat{N}} & \Ei\ar{r} & 0 \\
      & \Ei & & & \,.
    \end{tikzcd}\end{equation}
  The short exact sequence of (\ref{eq:diag-for-cup-2}) gives a distinguished triangle
    \[\K\to Cyl(\hat{u})\to C(\hat{u})\xrightarrow{\pr} \K[1]\]
  where $Cyl$ is the mapping cylinder, and the connecting homomorphism $\delta_y$ is represented by $\pr\colon C(\hat{u})\to \K[1]$. Hence we have a map
    \[\hat{f}[1]\circ \pr\colon C(\hat{u})\to \K[1]\to \Ei[1]\,.\]
  The map $q(\hat{u}):=\begin{pmatrix} 0 & \hat{N} \end{pmatrix}\colon C(\hat{u})\to \Ei$ is a quasi-isomorphism (see e.g. \cite[III.3.5]{Gelfand--Manin}) and we get the zig-zag
    \[\begin{tikzcd}
      \Ei & C(\hat{u})\ar{l}[swap]{q(\hat{u})}\ar{r}{\pr} & \K[1]\ar{r}{\hat{f}[1]} & \Ei[1]\,.
    \end{tikzcd}\]
  Now we apply $\HOM(-,\C)$, where $\C$ is the complex as in equality (\ref{res2}), to obtain the zig-zag
    \begin{equation}\begin{tikzcd} \label{eq:berliner-zig-zag}
      \HOM(\Ei,\C)\ar{r}{q(\hat{u})^*} & \HOM(C(\hat{u}),\C) & \HOM(\K[1],\C)\ar{l}[swap]{\pr^*} & \HOM(\Ei[1],\C)\ar{l}[swap]{\hat{f}[1]^*}\,.
    \end{tikzcd}\end{equation}
  The map $q(\hat{u})^*$ is a quasi-isomorphism since $q(\hat{u})$ is a quasi-isomorphism of locally free sheaves.

\subsection*{Using flat cohomology with compact support}
Now by applying $\Gamma_c(\Ufl,-)$ (see Definition \ref{def:global}) to the above zig-zag we get (where we write $\Gamma_c$ for legibility)
    \[\begin{tikzcd} \label{eq:mainer-zig-zag}
      \Gamma_c(\HOM(\Ei,\C)) \ar{r}{\Gamma_c(q(\hat{u})^*)} & \Gamma_c(\HOM(C(\hat{u}),\C)) & \Gamma_c(\HOM(\K[1],\C)\ar{l}[swap]{\pr^*}) & \Gamma_c(\HOM(\Ei[1],\C)\ar{l}[swap]{\hat{f}[1]^*})\,.
    \end{tikzcd}\]
We will show that the above zig-zag represents the map  $c_y^\sim\colon H^1_c(\Ufl,\mmu_n)\to H^2_c(\Ufl,\mmu_n)$ in cohomology. Note that we have a commutative diagram
        \begin{equation}\adjustbox{scale=0.9,center}{\begin{tikzcd}[sep=large]  \label{eq:three-isos}
      \Gamma_c(\HOM(\Ei,\C)) \ar{r}{\Gamma_c(q(\hat{u})^*)} \ar{d}{s} & \Gamma_c(\HOM(C(\hat{u}),\C)) \arrow{d}{t}& \Gamma_c(\HOM(\K[1],\C)\ar{l}[swap]{\Gamma_c(\pr^*)})\arrow{d}& \Gamma_c(\HOM(\Ei[1],\C)\ar{l}[swap]{\Gamma_c(\hat{f}[1]^*)}) \arrow{d}
      \\
       R \Gamma_c(\HOM(\Ei,\C))\ar{r}{R\Gamma_c(q(\hat{u})^*)}\ar{d}{s'} & R\Gamma_c(\HOM(C(\hat{u}),\C))\ar{d}{t'} & R \Gamma_c(\HOM(\K[1],\C)\ar{l}[swap]{R\Gamma_c(\pr^*)})\ar{d} & R\Gamma_c(\HOM(\Ei[1],\C)\ar{l}[swap]{R\Gamma_c(\hat{f}[1]^*)})\ar{d} \\
         R\hat{\Gamma}_c(\HOM(\Ei,\C)) \ar{r}{R\hat{\Gamma}_c(q(\hat{u})^*)} & R\hat{\Gamma}_c(\HOM(C(\hat{u}),\C)) & R\hat{\Gamma}_c(\HOM(\K[1],\C)\ar{l}[swap]{R\hat{\Gamma}_c(\pr^*)}) & R\hat{\Gamma}_c(\HOM(\Ei[1],\C)\ar{l}[swap]{R\hat{\Gamma}_c(\hat{f}[1]^*)})  \,.
    \end{tikzcd}}\end{equation}
The bottom zig-zag, from right to left, represents the Pontryagin dual $c_y^\sim$ of cup product with $y$. The idea is to compute $c_y^\sim$, by first moving up via the rightmost zig-zag, then moving left via the uppermost zig-zag, and finally moving down via the leftmost zig-zag. 
  \begin{lem}\label{lem:three-isos}
    The maps $s$, $t$, $t'$, and $\Gamma_c(C(q(\hat{u})^*))$ of (\ref{eq:three-isos}) give isomorphisms on $H^2$. Furthermore, the leftmost and rightmost vertical zig-zags gives isomorphisms on cohomology in degree 2.
  \end{lem}

  \begin{proof}
  See Appendix \ref{app:comp}.
  \end{proof}

  \begin{cor}\label{cor:dual-cup}
    The map $c_y^\sim\colon H^1_c(\Ufl,\mmu_n)\to H^2_c(\Ufl,\mmu_n)$ is given by \[H^1(\Gamma_c(q(\hat{u})^*))^{-1}\circ H^1(\Gamma_c(\pr^*))\circ H^1(\Gamma_c(\hat{f}[1]^*))\,.\]
  \end{cor}

  We are now in a position to compute the dual $c_y^\sim$ of the cup product 
    \[
      c_y=y\cupp- \colon H^1(U,\ZZ/n\ZZ)\to H^2(U,\ZZ/n\ZZ)\,.
    \] 
  Recall that any element $y\in H^1(U,\ZZ/n\ZZ)$ may be though of as an induced torsor and hence it can be represented by a cyclic degree $d|n$ extension $L/K$ which is unramified outside of $S$, together with a choice of generator $\sigma\in \Gal(L/K)$.

  We start by computing $c_y^\sim$ in the case where $S\neq \Omega_\RR$. By Lemma \ref{lem:dual-groups}, we may view $c_y^\sim \colon H^1_c(\Ufl,\mmu_n)\to H^2_c(\Ufl,\mmu_n)$ as a map $c_y^\sim\colon C_S(K)[n]\to C_S(K)/nC_S(K)$ which we now describe:

  \begin{lem}\label{lem:cup}
    Suppose that $S\neq \Omega_\RR$. Let $y\in H^1(U,\ZZ/n\ZZ)$ be represented by a cyclic extension $L/K$ of degree $d|n$ together with a choice of generator $\sigma\in \Gal(L/K)$. Then in view of Lemma \ref{lem:dual-groups}, we have that
      $$ c_y^\sim \colon C_S(K)[n]  \to C_S(K)/nC_S(K)
      $$
      is the map taking
      $$ \alpha  \mapsto \alpha^{n^2/2d}N_{L|K}(\beta)^{n/d},$$
    where $\beta$ is an element in $I_L$ such that $\alpha^{n/d}=t\beta/\sigma(\beta)$ in $I_L$, where $t\in L^\times$ satisfies $N_{L|K}(t)=\alpha^{-n}$ in $I_K$. Here $K^{\times}$ and $L^{\times}$ are embedded diagonally in $I_K$ and $I_L$, respectively.
  \end{lem}

  Before proving this, let us state the immediate consequence:

  \begin{thm}\label{thm:cup}
    Suppose that $S\neq \Omega_\RR$. Let $y$ and $z$ be elements in $H^1(U,\ZZ/n\ZZ)$. Suppose that $y$ is represented by a cyclic extension $L/K$, unramified outside of $S$, together with a choice of generator $\sigma \in \Gal(L/K)$, and assume that $L/K$ has degree $d|n$. Then under the identifications $H^1(U,\ZZ/n\ZZ)\cong (C_S(K)/nC_S(K))^\sim$ and $H^2(U,\ZZ/n\ZZ)\cong (C_S(K)[n])^\sim$ we have that $y\cupp z\in (C_S(K)[n])^\sim$ satisfies the formula
      \[
        \langle y\cupp z,\alpha\rangle=\langle z,\alpha^{n^2/2d}N_{L|K}(\beta)^{n/d}\rangle
      \]
    where $\langle -,- \rangle$ is the evaluation map and where $\beta$ is an element in $I_L$ such that $\alpha^{n/d}=t\beta/\sigma(\beta)$ in $I_L$ where $t\in L^\times$ satisfies $N_{L|K}(t)=\alpha^{-n}$ in $I_K$. Here $K^{\times}$ and $L^{\times}$ are embedded diagonally in $I_K$ and $I_L$, respectively.  
  \end{thm}

  \begin{proof}[Proof of Lemma \ref{lem:cup}]
    First recall that if $K$ is totally imaginary or $n$ is odd, we need not take the infinite places into account.

    We will use Corollary \ref{cor:dual-cup} and thus compute  $H^1(\Gamma_c(q(\hat{u})^*))^{-1}\circ H^1(\Gamma_c(\pr^*))\circ H^1(\Gamma_c(\hat{f}[1]^*))$.
    Let $\alpha\in C_S(K)[n]$ be represented by $(b,\bb,\alpha_S)\in \ker d^0\subseteq K^\times\oplus \Div(U)\oplus (\prod_{p\in S}K_p^\times\oplus \prod_{\nu\in \Omega_\RR}K_\nu^\times)$ where $d^0$ is as in the proof of Lemma \ref{lem:dual-groups}.
    Let $S_f'$ be the set of finite places in $L$ lying above $S_f$ and let $\Omega'_\RR$ be the real places of $L$. We apply $\Gamma_c$ to the zig-zag (\ref{eq:berliner-zig-zag}) and we obtain:

      \begin{equation}\label{eq:big}\adjustbox{scale=0.8,center}{
      \begin{tikzcd}
      K^\times\ar{r}\ar{d} &
      \begin{matrix}
        K^\times \\ \oplus \\ \Div(U) \\ \oplus \\ \prod_{p\in S_f}K_p^\times \\ \oplus \\ \prod_{\nu\in \Omega_\RR}K_\nu^\times
      \end{matrix}\ar{r}\ar{d} &
      \begin{matrix}
        \Div(U) \\ \oplus \\ \prod_{p\in S_f}K_p^\times \\ \oplus \\ \prod_{\nu\in \Omega_\RR}K_\nu^\times
      \end{matrix}\ar{d}{\Gamma_c(q(\hat{u})^*)^2}
      \\
      (L^\times)^{n/d}\ar{r}{d_{\Gamma_c}^0} &
      \begin{matrix}
        (L^\times)^{n/d} \\ \oplus \\ (L^\times)^{n/d} \\ \oplus \\ \Div(Y)^{n/d} \\ \oplus \\ \prod_{p\in S_f'}(L_p^\times)^{n/d} \\ \oplus \\ \prod_{\nu\in \Omega'_\RR}(L_\nu^\times)^{n/d}
      \end{matrix}\ar{r}{d_{\Gamma_c}^1} &
      \begin{matrix}
        K^\times \\ \oplus \\ (L^\times)^{n/d} \\ \oplus \\ \Div(Y)^{n/d} \\ \oplus \\ \Div(Y)^{n/d} \\ \oplus \\ \prod_{p\in S_f'}(L_p^\times)^{n/d} \\ \oplus \\ \prod_{\nu\in \Omega'_\RR}(L_\nu^\times)^{n/d} \\ \oplus \\ \prod_{p\in S_f'}(L_p^\times)^{n/d} \\ \oplus \\ \prod_{\nu\in \Omega'_\RR}(L_\nu^\times)^{n/d}
      \end{matrix}\ar{r} &
      \begin{matrix}
        K^\times \\ \oplus \\ \Div(U) \\ \oplus \\ \Div(Y)^{n/d} \\ \oplus \\ \prod_{p\in S_f}K_p^\times \\ \oplus \\ \prod_{\nu\in \Omega_\RR}K_\nu^\times \\ \oplus \\ \prod_{p\in S_f'}(L_p^\times)^{n/d} \\ \oplus \\ \prod_{\nu\in \Omega'_\RR}(L_\nu^\times)^{n/d}
      \end{matrix}\ar{r} &
      \begin{matrix}
        \Div(U) \\ \oplus \\ \prod_{p\in S_f}K_p^\times \\ \oplus \\ \prod_{\nu\in \Omega_\RR}K_\nu^\times
      \end{matrix} \\
       & \begin{matrix}\prod_{p\in S_f}K_p^\times \\ \oplus \\ \prod_{\nu\in \Omega_\RR}K_\nu^\times\end{matrix}\ar{r}\ar{u} &
      \begin{matrix}
        K^\times \\ \oplus \\ (L^\times)^{n/d} \\ \oplus \\ \Div(Y)^{n/d} \\ \oplus \\ \prod_{p\in S_f'}(L_p^\times)^{n/d} \\ \oplus \\ \prod_{\nu\in \Omega'_\RR}(L_\nu^\times)^{n/d}
      \end{matrix}\ar{r}\ar{u}{\Gamma_c(\pr^*)^1} &
      \begin{matrix}
        K^\times \\ \oplus \\ \Div(U) \\ \oplus \\ \Div(Y)^{n/d} \\ \oplus \\ \prod_{p\in S_f}K_p^\times \\ \oplus \\ \prod_{\nu\in \Omega_\RR}K_\nu^\times \\ \oplus \\ \prod_{p\in S_f'}(L_p^\times)^{n/d} \\ \oplus \\ \prod_{\nu\in \Omega'_\RR}(L_\nu^\times)^{n/d}
      \end{matrix}\ar{r}\ar{u} &
      \begin{matrix}
        \Div(U) \\ \oplus \\ \prod_{p\in S_f}K_p^\times \\ \oplus \\ \prod_{\nu\in \Omega_\RR}K_\nu^\times
      \end{matrix}\ar{u} \\
      & K^\times\ar{r}\ar{u} &
      \begin{matrix}
        K^\times \\ \oplus \\ \Div(U) \\ \oplus \\ \prod_{p\in S_f}K_p^\times \\ \oplus \\ \prod_{\nu\in \Omega_\RR}K_\nu^\times
      \end{matrix}\ar{r}\ar{u}{\Gamma_c(\hat{f}^*)^1[1]} &
      \begin{matrix}
        \Div(U) \\ \oplus \\ \prod_{p\in S_f}K_p^\times \\ \oplus \\ \prod_{\nu\in \Omega_\RR}K_\nu^\times
      \end{matrix}\ar{u}
      \end{tikzcd}}\end{equation}
    where
    \[\begin{split}
      \Gamma_c(\hat{f}^*)^1[1](b,\bb, a) & = (b, i(b), i(a)) \\
      \Gamma_c(\pr^*)^1(b',b'',I,a') & = (b', b'', I, 0, a', 1)  \\
      \Gamma_c(q(\hat{u})^*)^2(J,a'') & = (1, 1, 0, i(J), 1, i(a''))\\
      d_{\Gamma_c}^1 & = \begin{pmatrix}
          N_{L|K} & 0 & 0 & 0 \\
          n & 1-\sigma' & 0 & 0 \\
          -\divis & 0 & 1-\sigma' & 0\\
          0 & -\divis & -n & 0 \\
          -\eta & 0 & 0 & 1-\sigma' \\
          0 & -\eta & 0 & -n
      \end{pmatrix}\,.
    \end{split}\]

    Here $a$ is an element in $\prod_{p\in S_f}K_p^\times \oplus \prod_{\nu\in \Omega_\RR}K_\nu^\times$ and $i$ is the canonical map
      \[
        i\colon \prod_{p\in S_f}K_p^\times \oplus \prod_{\nu\in \Omega_\RR}K_\nu^\times\to
        \prod_{p\in S_f'}(L_p^\times)^{n/d} \oplus \prod_{\nu\in \Omega'_\RR}(L_\nu^\times)^{n/d}\,.
      \]
    By $\Omega'_\RR$ we mean the set of places over $\Omega_\RR$, including the ones that ramify. 
    Note that in the matrix description of $d_{\Gamma_c}^1$, every summand of the form
      \[
        \begin{matrix}
          \prod_{p\in S_f'}(L_p^\times)^{n/d} \\ \oplus \\ \prod_{\nu\in \Omega'_\RR}(L_\nu^\times)^{n/d}
        \end{matrix}
      \]
    is treated as a \emph{single} summand.
    Furthermore, by abuse of notation, $\eta$ is the canonical map
      \[
        \eta\colon (L^\times)^{n/d} \to \prod_{p\in S_f'}(L_p^\times)^{n/d} \oplus \prod_{\nu\in \Omega'_\RR}(L_\nu^\times)^{n/d}\,,
      \]

      \[(\sigma'-1)(a)=\begin{pmatrix}\sigma(a_{n/d})/a_1 & a_1/a_2 & \dots & a_{(n/d)-1}/a_{n/d}\end{pmatrix}\,\]
    and
      \[N_{Y|U}(a)=\prod_{i=1}^{n/d}\prod_{j=0}^{d-1}\sigma(a_i)^j\,.\]
    Here $(a_1,\dots, a_n)$ is an element in $(L^\times)^{n/d}$, $\Div(Y)^{n/d}$, or $(\prod_{p\in S_f'}L_p^\times \oplus \prod_{\nu\in \Omega'_\RR}L_\nu^\times)^{n/d}$. 

    The element $(b,\bb,\alpha_S)\in \ker d^0\subseteq K^\times\oplus \Div(U)\oplus (\prod_{p\in S_f}K_p^\times\oplus \prod_{\nu\in \Omega_\RR}K_\nu^\times)$ representing $\alpha \in C_S(K)[n]$ is sent to
      \begin{equation}\label{eq:red-1}\begin{pmatrix}
        b & i(b) & i(\bb) & 0 & i(\alpha_S) & 1
      \end{pmatrix}\end{equation}
    via $\Gamma_c(\pr^*)^1\circ \Gamma_c(\hat{f}^*)^1$. We need to reduce this modulo the image of $d_{\Gamma_c}^1$ to get an element of the form
      \[\Gamma_c(q(\hat{u})^*)^2(J,a') = \begin{pmatrix} 1 & 1 & 0 & i(J) & 1 & i(a')\end{pmatrix}\,.\] The image of $\alpha$ under the connecting homomorphism is then given by the class of  $(J,a')$ in $C_S(K)/n.$
    We have seen in Section \ref{sec:cohomology-groups} seen that the canonical surjection
    \[
      C_S(K)\to (\Div U\oplus \prod_{p\in S}K_p^\times\oplus \prod_{\nu\in \Omega_\RR}K^\times_\nu/2K^\times_\nu)/K^\times
    \]
    induces an isomorphism on $n$-torsion.

    We have that $N_{L|K}(i(\alpha^{n/d}))=(\alpha^{n/d})^d=1$ in $C_S(K)$, since $\alpha$ is $n$-torsion. But
      \[H^{1}(G_{L/K},C_S(L))=\ker N_{L|K}/(\sigma-1)C_S(L)=0\]
    (see e.g. \cite[p.~620]{Neukirch--Schmidt--Wingberg}) and hence there exists a $\beta=(I,\beta)\in \Div(Y)\oplus (\prod_{p\in S_f}L_p^\times\oplus \prod_{\nu\in \Omega'_\RR}L_\nu^\times)$ and a $t\in L^\times$ such that
      \[\begin{split}
        \bb^{n/d}\OO_L & = \divis(t)I/\sigma(I)\,,\\
        i(\alpha_S^{n/d}) & = \eta(t)\beta/\sigma(\beta)\,.
      \end{split}\]
    Note here that $\Omega'_\RR$ includes complex places over $\Omega_\RR$. Taking the norm on both sides of the latter equality we get $\alpha_S^n=N_{L|K}(\eta(t))$ and since $\alpha_S^{n}=\eta(b^{-1})$ we get that
      \[N_{L|K}(t)=b^{-1}\,.\]
    We now put
      \[\begin{pmatrix}
        \underline{I} \\
        \underline{\beta}
      \end{pmatrix} = \begin{pmatrix}
        I & \bb I & \bb^2I & \dots & \bb^{\frac{n}{d}-1}I \\
        \beta & \alpha_S\beta & \alpha_S^2\beta & \dots & \alpha_S^{\frac{n}{d}-1}\beta
      \end{pmatrix}\in \begin{matrix}
        \Div(L)^{n/d} \\ \oplus \\ \left(\prod_{p\in S_f}L_p^\times\oplus \prod_{\nu\in \Omega'_\RR}L_\nu^\times\right)^{n/d}
      \end{matrix}\]
    which satisfies
      \[\begin{split}
        (\sigma'-1)(\underline{I}) & = \begin{pmatrix}
          \divis(t)\bb^{-1} & \bb^{-1} & \bb^{-1} & \dots & \bb^{-1}
        \end{pmatrix} \\
        (\sigma'-1)(\underline{\beta}) & = \begin{pmatrix}
          \eta(t)\alpha_S^{-1} & \alpha_S^{-1} & \alpha_S^{-1} & \dots & \alpha_S^{-1}
        \end{pmatrix}\,.
      \end{split}\]
    If we reduce (\ref{eq:red-1}) modulo
      \[\begin{pmatrix}
          N_{L|K} & 0 & 0 & 0 \\
          n & 1-\sigma' & 0 & 0 \\
          -\divis & 0 & 1-\sigma' & 0\\
          0 & -\divis & -n & 0 \\
          -\eta & 0 & 0 & 1-\sigma' \\
          0 & -\eta & 0 & -n
      \end{pmatrix}\begin{pmatrix}
        \underline{t}^{-1} \\ 1 \\ \underline{I} \\ \underline{\beta}
      \end{pmatrix}= \begin{pmatrix}
        N_{L|K}(\underline{t}^{-1}) \\
        \underline{t}^{-n} \\
        \divis(\underline{t})\underline{I}/\sigma'(\underline{I}) \\
        \underline{I}^{-n} \\
        \eta(\underline{t})\underline{\beta}/\sigma'(\underline{\beta}) \\
        \underline{\beta}^{-n}
      \end{pmatrix} = \begin{pmatrix}
        b \\
        \underline{t}^{-n} \\
        \bb \\
        \underline{I}^{-n} \\
        \alpha_S \\
        \underline{\beta}^{-n}
      \end{pmatrix}\]
    where
      \[\underline{t}=\begin{pmatrix}
        t & 1 & 1 & \dots & 1
      \end{pmatrix}\,,\]
    then we get
      \[\begin{pmatrix}
        b \\ b \\ \bb \\ 1 \\ \alpha_S \\ 1
      \end{pmatrix}+\begin{pmatrix}
        b^{-1} \\
        \underline{t}^{n} \\
        -\bb \\
        \underline{I}^n \\
        \alpha_S^{-1} \\
        \underline{\beta}^{n}
      \end{pmatrix}=\begin{pmatrix}
        1 \\
        b\underline{t}^{n} \\
        1 \\
        \underline{I}^n \\
        1 \\
        \underline{\beta}^{n}
      \end{pmatrix}\]
    since $bN_{L|K}(t)=b\alpha_S^{n}=1$.

    As in the proof of \cite[Lemma 3.7]{Ahlqvist--Carlson-cup} we now consider the element
      \[w=\prod_{j=0}^{d-1}\sigma^j(t^{n-jn/d})\]
    and put
      \[\underline{w}=\begin{pmatrix}
        w & bw & b^2w & \dots & b^{n/d-1}w
      \end{pmatrix}\,.\]
    We have
        \[\begin{pmatrix}
            N_{L|K} & 0 & 0 & 0 \\
            n & 1-\sigma' & 0 & 0 \\
            -\divis & 0 & 1-\sigma' & 0\\
            0 & -\divis & -n & 0 \\
            -\eta & 0 & 0 & 1-\sigma' \\
            0 & -\eta & 0 & -n
        \end{pmatrix}\begin{pmatrix}
          1 \\ \underline{w}^{-1} \\ 1 \\ 1
        \end{pmatrix}=\begin{pmatrix}
          1 \\ b^{-1}\underline{t}^{-n} \\ 1 \\  \divis(w) \\ 1 \\ \eta(w)
        \end{pmatrix}\]
    since
        \[\begin{split}
          (\sigma'-1)(\underline{w}) & = \begin{pmatrix}
            \sigma(wb^{n/d-1})w^{-1} & b^{-1} & b^{-1} & \dots & b^{-1}
          \end{pmatrix} \\
          & = \begin{pmatrix}
            b^{-1}t^{-n} & b^{-1} & b^{-1} & \dots & b^{-1}
          \end{pmatrix}
        \end{split}\]
    Hence we get in cohomology
        \[\begin{pmatrix}
          1 \\
          b\underline{t}^{n} \\
          1 \\
          \underline{I}^n \\
          1 \\
          \underline{\beta}^{n}
        \end{pmatrix}= \begin{pmatrix}
          1 \\
          b\underline{t}^{n} \\
          1 \\
          \underline{I}^n \\
          1 \\
          \underline{\beta}^{n}
        \end{pmatrix}+\begin{pmatrix}
          1 \\ b^{-1}\underline{t}^{-n} \\ 1 \\ \divis(w) \\ 1 \\ \eta(w)
        \end{pmatrix}=\begin{pmatrix}
          1 \\
          1 \\
          1 \\
          \underline{I}^n\divis(\underline{w})\\
          1 \\
          \underline{\beta}^{n}\eta(\underline{w})
        \end{pmatrix}\]
    We have
     \[\begin{split}
       \underline{I}^n\divis(\underline{w}) & = \begin{pmatrix}
         I^n\divis(w) & I^n\divis(w) & \dots & I^n\divis(w)
       \end{pmatrix}\,, \\
       \underline{\beta}^n\eta(\underline{w}) & = \begin{pmatrix}
         \beta^n\eta(w) & \beta^n\eta(w) & \dots & \beta^n\eta(w)
       \end{pmatrix}\,,
     \end{split}\]
    and since
      \[\begin{split}
        I^n\divis(w) & = I^n\prod_{j=0}^{d-1}\sigma^j(\divis(t))^{n-jn/d} \\
                     & = I^n\prod_{j=0}^{d-1}\sigma^j(\bb^{n/d}I^{-1}\sigma(I))^{n-jn/d} \\
                     & = I^n\bb^{\frac{n(d+1)}{2}\frac{n}{d}}I^{-n}(I\sigma(I)\dots\sigma^{d-1}(I))^{\frac{n}{d}} \\
                     & = \bb^{\frac{n(d+1)}{2}\frac{n}{d}}N_{L|K}(I)^{\frac{n}{d}}
      \end{split}\]
    and similarly
    \[\begin{split}
      \beta^n\eta(w) & = \beta^n\prod_{j=0}^{d-1}\sigma^j(\eta(t))^{n-jn/d} \\
                   & = \beta^n\prod_{j=0}^{d-1}\sigma^j(\alpha_S\sigma(\beta)/\beta)^{n-jn/d} \\
                   & = \beta^n\alpha_S^{\frac{n(d+1)}{2}\frac{n}{d}}\beta^{-n}N_{L|K}(\beta)^{n/d} \\
                   & = \alpha_S^{\frac{n(d+1)}{2}\frac{n}{d}}N_{L|K}(\beta)^{n/d}\,,
    \end{split}\]
    we get that
    \[\begin{pmatrix}
      1 \\
      1 \\
      1 \\
      \underline{I}^{n}\divis(\underline{w})\\
      1 \\
      \underline{\beta}^{n}\eta(\underline{w})
    \end{pmatrix}\sim \begin{pmatrix}
      1 \\
      1 \\
      1 \\
      i\left(\bb^{\frac{n(d+1)}{2}\frac{n}{d}}N_{L|K}(I)^{\frac{n}{d}}\right)\\
      1 \\
      i\left(\alpha_S^{\frac{n(d+1)}{2}\frac{n}{d}}N_{L|K}(\beta)^{n/d}\right)
    \end{pmatrix}\]
  This completes the proof.
  \end{proof}

Finally we are ready to compute the cup product for the case $S=\Omega_\RR$.  We proceed by computing the cup product $H^1(X,\ZZ/n\ZZ)\times H^1(X,\ZZ/n\ZZ)\to H^2(X,\ZZ/n\ZZ)$ and then sketch how to compute the higher cup products.

  \begin{lem}\label{lem:cup2}
    Suppose that $S=\Omega_\RR$. Let $y\in H^1(X,\ZZ/n\ZZ)$ be represented by a cyclic extension $L/K$ of degree $d|n$ together with a choice of generator $\sigma\in \Gal(L/K)$. Then in view of Lemma \ref{lem:dual-groups}, we have that
      \[c_y^\sim \colon Z^1/B^1  \to \Cl^+(K)/n\Cl^+(K) \]
      is the map
 	\[ (b,\bb)  \mapsto \bb^{n^2/2d}N_{L|K}(I)^{n/d}\,,\]
    where $I$ is an element in $\Div(L)$ such that $\bb^{n/d}=\divis(t)I/\sigma(I)$, where $t\in L^\times_+$ satisfies $N_{L|K}(t)=b^{-1}$.
  \end{lem}

  The following theorem is an immediate consequence:

  \begin{thm}\label{thm:cup2}
    Suppose that $S= \Omega_\RR$. Let $y$ and $z$ be elements in $H^1(X,\ZZ/n\ZZ)$. Suppose that $y$ is represented by a cyclic extension $L/K$, unramified outside of $S$, together with a choice of generator $\sigma \in \Gal(L/K)$, and assume that $L/K$ has degree $d|n$. Then under the identifications $H^1(X,\ZZ/n\ZZ)\cong (\Cl^+(K)/n\Cl^+(K))^\sim$ and $H^2(X,\ZZ/n\ZZ)\cong (Z^1/B^1)^\sim$ we have that $y\cupp z\in (Z^1/B^1)^\sim$ satisfies the formula
      \[
        \langle y\cupp z,(b,\bb)\rangle=\langle z,\bb^{n^2/2d}N_{L|K}(I)^{n/d}\rangle
      \]
    where $I$ is an element in $\Div(L)$ such that $\bb^{n/d}=\divis(t)I/\sigma(I)$, where $t\in L^\times_+$ satisfies $N_{L|K}(t)=b^{-1}$.
  \end{thm}

  \begin{proof}[Proof of Lemma \ref{lem:cup2}]
    Since the case where $K$ is totally imaginary or $n$ is odd is covered in \cite{Ahlqvist--Carlson-cup}, we may assume that $K$ is not totally imaginary and that  $n$ is even. We have $H^1_c(X_\fl,\mu_n)\cong Z^1/B^1$
    where
      \[
          \begin{split}
            Z^1 & = \{(b,\bb)\in K_+^\times\oplus \Div X : \divis(b)\bb^n=1\}\,, \\
            B^1 & = \{(a^{-n},\divis(a)):a\in K^\times_+\}\,.
          \end{split}
      \]

    The proof is very similar to the proof of \cite[Lemma 3.7]{Ahlqvist--Carlson-cup}. We will again consider the zig-zag represented by the big diagram \ref{eq:big}. Again, $(b,\bb,1)\in \ker d^0\subseteq K^\times\oplus \Div(K)\oplus \prod_{\nu\in \Omega_\RR}K_\nu^\times$ is sent to
      \begin{equation}\label{eq:red-2}\begin{pmatrix}
        b & i(b) & i(\bb) & 0 & 1 & 1
      \end{pmatrix}\end{equation}
    via $\Gamma_c(\pr^*)^1\circ \Gamma_c(\hat{f}^*)^1$. We need to reduce this modulo the image of $d_{\Gamma_c}^1$ to get an element of the form
      \[\Gamma_c(q(\hat{u})^*)^2(J,a') = \begin{pmatrix} 1 & 1 & 0 & i(J) & 1 & i(a')\end{pmatrix}\,.\]

    The pair $(\bb^{n/d}\OO_L,i(\eta(b^{-1/d})))$ defines a class in $C_S(L)$ (this time including complex places) which lies in the kernel of $N_{L|K}\colon C_S(L)\to C_S(K)$ since $i(\eta(b^{-1/d}))$ has norm $\eta(b^{-1})$. But
      \[
        H^{1}(G_{L/K},C_S(L))=\ker N_{L|K}/(\sigma-1)C_S(L)=0
      \]
    and hence there exists a pair $(\bb',\tilde{\beta})\in \Div(L)\oplus \prod_{\nu\in \Omega'_\RR}L_\nu^\times$ and $a\in L^\times$ such that
      \[\begin{split}
        \bb^{n/d}\OO_L & = \divis(a)\bb'/\sigma(\bb')\,, \\
        i(\eta(b^{-1/d})) & = \eta(a)\tilde{\beta}/\sigma(\tilde{\beta})\,.
      \end{split}\,,\]
    for some $a\in L^\times_+$. 
    Taking the norm on both sides we get $\divis(b^{-1})=\bb^n=\divis(N_{L|K}(a))$ and hence $N_{L|K}(a)=b^{-1}u$ for some unit $u\in K^\times$. From the second equation we get $\eta(b^{-1})=N_{L|K}(\eta(a))$ and hence $N_{L|K}(a)$ is totally positive. Since $b^{-1}$ is totally positive we get that $u$ is totally positive.

    Units are always norms in unramified extensions of local fields and $u$ is totally positive. Hence Hasse's norm theorem \cite{HasseNorm} implies that there is a $v\in L^\times$ such that $u^{-1}=N_{L|K}(v)$. Since $N_{L|K}(\divis(v))$ is the unit ideal, Hilbert 90 for ideals implies that there is a $J\in \Div(L)$ such that $\divis(v)=J/\sigma(J)$. We now put $t=av$ and $I=\bb'/J$. Then
      \[
        \bb^{n/d}\OO_L=\divis(av)(\bb'/J)/(\sigma(\bb'/J))=\divis(t)I/\sigma(I)\,.
      \]
    Note that $N_{L|K}(\eta(i(b^{-1/d})t^{-1}))=1$. By Hilbert's theorem 90 for idèles, we have
      \[
        H^1(G_{L/K},I_L)\cong \ker N_{L|K}/(1-\sigma)I_L=1
      \]
    and hence there exists a $\xi\in I_L$ (with a one in every non-archimedean component) such that $\eta(i(b^{-1/d})t^{-1})=\xi/\sigma(\xi)$. Put $\beta'=\beta\xi$, where $\beta\in \prod_{\nu\in \Omega'_\RR}L_\nu^\times$ is the class represented by $\tilde{\beta}$.

    We now put
      \[\begin{pmatrix}
        \underline{I} \\
        \underline{\beta}
      \end{pmatrix} = \begin{pmatrix}
        I & \bb I & \bb^2I & \dots & \bb^{\frac{n}{d}-1}I \\
        \beta' & \beta' & \beta' & \dots & \beta'
      \end{pmatrix}\in \begin{matrix}
        \Div(L)^{n/d} \\ \oplus \\  (\prod_{\nu\in \Omega'_\RR}L_\nu^\times)^{n/d}
      \end{matrix}\]
    and the rest follows word by word as the remainder of the proof of Lemma \ref{lem:cup} starting from page 16.
  \end{proof}

 Let us summarize what we know so far about the ring structure of $H^*(U,\ZZ/n\ZZ)$. First note that if $K$ is totally imaginary or $n$ is odd, then the picture is complete since the case $U=X$ was treated in \cite{Ahlqvist--Carlson-cup} and the case $U\neq X$ is completely determined by Theorem \ref{thm:cup}.

  When $K$ has real places and $n$ is even, we have that $H^i(U,\ZZ/n\ZZ)\neq 0$ for $i\geq 3$. Hence it remains to compute the cup products $H^i(U,\ZZ/n\ZZ)\times H^j(U,\ZZ/n\ZZ)\to H^{i+j}(U,\ZZ/n\ZZ)$ for $i + j \geq 3$.
  For $i\geq 3$, the restriction map $H^i(U,\ZZ/n\ZZ)\to \prod_{\nu\in \Omega_\RR}H^i(\Spec K_\nu,\ZZ/n\ZZ)$ is an isomorphism and the right hand side is the group cohomology of $\ZZ/n\ZZ$ viewed as a $\Gal(\CC/\RR)$-module. The cup product can then be determined as follows:
  For $i$ and $j$ such that $i+j \geq 3$, let $(x,y)\in H^i(U,\ZZ/n\ZZ)\times H^j(U,\ZZ/n\ZZ)$ and let $\rho\colon H^*(U,\ZZ/n\ZZ)\to \prod_{\nu\in \Omega_\RR}H^*(\Spec K_\nu,\ZZ/n\ZZ)$ be the restriction.
  Then we have that $x\cupp y$ is the unique element in $H^{i+j}(U,\ZZ/n\ZZ)$ restricting to $r(x)\cupp r(y)\in \prod_{\nu\in \Omega_\RR}H^{i+j}(\Spec K_\nu,\ZZ/n\ZZ)$.
  We know that $H^1(U,\ZZ/n\ZZ)$ classifies torsors and $H^2(U,\ZZ/n\ZZ)$ classifies gerbes. The restriction to a real place then simply tests if the torsor or gerbe is trivial or not over that place.
  Hence we have now determined the whole ring structure of $H^*(U,\ZZ/n\ZZ)$, by essentially arguing that it reduces to group cohomology in the degrees we have not discussed in detail in this section.

\section{Applications}\label{sec:examples}

\subsection*{Class field tower groups}
Let $K$ be a number field and $S$ a finite set of places of $K$ (including the infinite places). Fix a prime number $p$ and let $G=G_S$ be the Galois group of the maximal pro $p$-extension of $K$, unramified outside of $S$. Choose a minimal presentation by free pro $p$-groups 
  \[
    1\to R \to F \to G \to 1  
  \]
and let $d(G)$ and $r(G)$ be the least number of generators of $F$ and of $R$ as a normal subgroup of $F$ respectively. Then we have $d(G)=H^1(G, \ZZ/p\ZZ)$ and $r(G)=H^2(G, \ZZ/p\ZZ)$. Let $F=F_1\supseteq F_2\supseteq F_3\supseteq \dots$ be the Zassenhaus filtration, choose a generating set $\Gen(R)\subseteq R$ and define $r_k$ to be the number of elements $g\in \Gen(R)$ for which $k$ is the largest integer such that $g\in F_k\setminus F_{k+1}$. It is possible to choose $\Gen(R)$ in such a way that the \emph{Zassenhaus polynomial}
  \[
    Z_G(t) = 1-dt+\sum_{i=2}^{\infty} r_kt^k   
  \]
becomes an invariant of $G$ (see e.g. \cite{Mcleman}). An improved version (see \cite{Koch-Galois}) of the celebrated result of Golod--Shafarevich states that if $G$ is finite then $Z_G(t)>0$ for all $t\in (0,1)$. This result becomes even more powerful if it possible to determine the numbers $r_k$. By \cite[Theorem 7.23]{Koch-Galois}, the cup products will determine $r_2$. Furthermore, we have natural maps $H^i(G, \ZZ/p\ZZ)\to H^i(U, \ZZ/p\ZZ)$ preserving the cup product, where the right hand side denotes the \'etale cohomology of the scheme $U=\Spec \OO_K\setminus\{S_f\}$. We will show that we may find $r_2$ by computing cup products via our formulas of Section \ref{sec:cup} (general case), our previous paper \cite{Ahlqvist--Carlson-cup} (proper case), or \cite{CarlsonSchlankUnramified} (proper case and $p=2$). 

\begin{ex}
  Let $K=\QQ(\sqrt{-3\cdot13\cdot17\cdot29\cdot113})$ and choose bases $\{x_i\}_{1\leq i\leq d(G)}\subseteq H^1(X, \ZZ/2\ZZ)$ and $\{b_k\}_{1\leq k\leq r_{\text{\'et}}(G)}\subseteq \Hom(H^2(X, \ZZ/2\ZZ), \ZZ/2\ZZ)$ where $r_{\text{\'et}}(G)$ is the rank of $H^2(X, \ZZ/2\ZZ)$. Let $G$ be the maximal unramified pro 2-extension of $K$. Using a C program that we have written using the library PARI \cite{PARI2}, we see that $\Cl(K)$ has cycle type $[16,4,4,2]$ and the $d(G)^2\times r_{\text{\'et}}(G)$ cup products matrix $(\langle x_i\cupp x_j, b_k\rangle)_{(i,j), k}$ has rank 3. We have $d(G)=\rk \Cl(K)/2\Cl(K)=4$ and since $\rk \OO_K^\times/2\OO_K^\times=1$, we se that $r(G)$ is either 4 or 5. This means that the Zassenhaus polynomial will be either $1-4t+3t^2+t^3$ or $1-4t+3t^2+2t^3$, which both have zeros in the interval $(0,1)$, and we conclude that $G$ must be infinite.  
\end{ex}

Let $G=P_0(G)\supseteq P_1(G)\supseteq P_2(G) \supseteq \dots$ be the lower $p$-central series of $G=G_S$ and define $Q_n(G)=G/P_n(G)$ (beware that we start our indexing from 0). For instance we have $Q_0(G)=1$ and $Q_1(G)=G/G^p[G,G] \cong C_S(K)/pC_S(K)$. Let $d=d(G)$, $r'=r_{\text{\'et}}(G)\geq r(G)$, and choose bases $\{x_i\}_{1\leq i\leq d}\subseteq H^1(U, \ZZ/p\ZZ)$ and $\{b_k\}_{1\leq k\leq r'}\subseteq \Hom(H^2(U, \ZZ/p\ZZ), \ZZ/p\ZZ)\cong H^1_c(\Ufl, \mmu_p)$. By computing cup products in \'etale cohomology we find a presentation for $Q_2(G)$:

\begin{prop}
  Let $1\to R\to F\to G\to 1$ be a minimal presentation with $F$ a free pro-$p$-group generated by $a_1, \dots, a_d$ and let $x_1,\dots, x_d$ be the corresponding elements forming a basis for $H^1(U, \ZZ/p\ZZ)$. Then $Q_2(G)\cong Q_2(F)/C$ where $C$ is the normal subgroup generated by the set
    \[
      \left\{\prod_{1\leq i\leq j\leq d(G)}(a_i, a_j)^{-\langle c_{ij}, b_k\rangle}\right\}_{1\leq k\leq r_{\text{\'et}}(G)}\,,
    \] 
  where 
    \[
      (a_i,a_j) = \begin{cases}
        a_i^p\,, & i=j\,, \\
        a_ia_ja_i^{-1}a_j^{-1}\,, & i\neq j\,,
      \end{cases}  
    \] 
    \[
      c_{ij}=\begin{cases}
        B(x_i)\,, & i= j\,, \\
        x_i\cupp x_j\,, & i\neq j\,.
      \end{cases}  
    \]
  and $B\colon H^1(X, \ZZ/p\ZZ)\to H^2(X, \ZZ/p\ZZ)$ is the Bockstein homomorphism. 
\end{prop}

\begin{proof}
  Let $g_1, \dots, g_r$ be a generating set of elements of $R$ and let $\chi_1, \dots \chi_d$ be elements mapping to $x_1, \dots, x_d$ under the canonical isomorphism $H^1(G, \ZZ/p\ZZ)\cong H^1(U, \ZZ/p\ZZ)$. By \cite[Theorem 7.22, 7.23, 7.24]{Koch-Galois} we may write 
    \begin{equation}\label{eq:relations}
      g_k = \left(\prod_{1\leq i\leq d(G)}a_i^{-\varphi_k(B(x_i))p}\prod_{1\leq i< j\leq d(G)}(a_i, a_j)^{-\varphi_k(\chi_i\cupp \chi_j)}\right)g_k'
    \end{equation} 
  where $g_k'\in P_3(G)$ and $\varphi_k\colon H^2(G,\ZZ/p\ZZ)\to \ZZ/p\ZZ$ are defined by $\varphi_k(\alpha)=(\text{tra}^{-1}(\alpha))(g_k)$ with 
    \[
      \text{tra}\colon H^1(R, \ZZ/p\ZZ)^G\cong H^2(G, \ZZ/p\ZZ)\,.
    \]
  These elements $\varphi_k$ form a basis for the dual group $H^2(G,\ZZ/p\ZZ)^\vee$. Furthermore, every basis of $H^2(G,\ZZ/p\ZZ)^\vee$ arise in this way, so we may in fact take the basis $\{\varphi_k\}_k$ to be arbitrary and there will be set $\{g_1, \dots, g_r\}$ as above such that $\varphi_k(\alpha)=(\text{tra}^{-1}(\alpha))(g_k)$. 
  
  We have a canonical injection $H^2(G, \ZZ/p\ZZ)\hookrightarrow H^2(X, \ZZ/p\ZZ)$ which gives a surjection $H^2(U, \ZZ/p\ZZ)^\vee\to H^2(G,\ZZ/p\ZZ)^\vee$. Since we have an isomorphism in degree one: $H^1(G, \ZZ/p\ZZ)\cong H^1(X, \ZZ/p\ZZ)$ and the maps $H^i(G, \ZZ/p\ZZ)\to H^i(X, \ZZ/p\ZZ)$ commute with cup products in the obvious sense (see \cite[Section 3]{AhlqvistCarlson+2025}), the cup product $H^1(X, \ZZ/p\ZZ)\otimes H^1(X, \ZZ/p\ZZ)\to H^2(X, \ZZ/p\ZZ)$ takes values in the image of $H^2(G, \ZZ/p\ZZ)\hookrightarrow H^2(X, \ZZ/p\ZZ)$. The proposition now follows from noting that the following diagram commutes:
    \[
      \begin{tikzcd}
        H^2(U,\ZZ/p\ZZ)\times H^{2}_c(\Ufl,\mmu_p)\ar{r}{\langle-,-\rangle}\ar[shift left=15, twoheadrightarrow]{d}{} & H^{3}_c(\Ufl,\mmu_p)\ar{d}{\inv}[swap]{\cong} \\
        H^2(U,\ZZ/p\ZZ)\times H^2(U,\ZZ/p\ZZ)^\vee\ar{r}{\text{eval}}\ar[shift left=15, twoheadrightarrow]{d}{}\ar[shift left=15, equal]{u} & \ZZ/p\ZZ\ar[equal]{d} \\
        H^2(G,\ZZ/p\ZZ)\times H^{2}(G,\ZZ/p\ZZ)^\vee\ar{r}{\text{eval}}\ar[shift left=15, hookrightarrow]{u} & \ZZ/p\ZZ
      \end{tikzcd}
    \]
  commutes. The upper diagram commutes by Artin--Verdier duality and the lower diagram obviously commutes. Now just choose lifts $b_1,\dots,b_r\in H^{2}_c(\Ufl,\mmu_p)$ of the elements $\varphi_1,\dots, \varphi_r$ and extend them to a basis of $H^{2}_c(\Ufl,\mmu_p)$. This gives the statement of the proposition. 
\end{proof}

\begin{ex}
The field $K=\QQ(\sqrt{-5460})$ is of great interest since if it has infinite class field tower it would decrease the Odlyzko root discriminant bound. Martinet have conjectured that all imaginary quadratic fields with 2-class group of rank at least 4 (including this example) have infinite 2-class field tower. On the other hand, Boston--Wang has conjectured that $K$ has finite 2-class field tower \cite{Boston--Wang}. Using our program we see that $Q_2(G)$ has order $2^9$ and presentation
  \[
    \langle a,b,c,d : (a,b)(a,c)(b,c), a^2(a,c)b^2(b,d), (a,c)b^2c^2, a^2(a,b)(c,d), a^2(a,c)(a,d)b^2d^2, R_3\rangle\,,  
  \]
where $R_3=P_2(\Free(a,b,c,d))$ (indexing from 0, so $P_0(F)=F$, $P_1(F)=F^p[F,F]$ etc.).
\end{ex}

\begin{ex}
Similarly to the previous example, we may consider the real quadratic field $K=\QQ(\sqrt{1155})$. Using our program we see that $Q_2(G)$ has order $2^9$ and presentation
  \[
    \langle a,b,c,d : (a,c)b^2, a^2(a,c)(a,d)(b,c), (a,b)(a,d)(b,d)c^2, a^2(a,c)(b,d)(c,d), d^2, R_3\rangle\,,  
  \]
where $R_3=P_2(\Free(a,b,c,d))$.
\end{ex}

\subsection*{Comparison with the Legendre symbol}
We will now show how to view the classical Legendre symbol $\legendre{p}{q}$ for primes $p,q \equiv 1$ (mod 4) as a cup product.

\begin{prop} \label{legendre}
  Let $p$ and $q$ be distinct odd primes both equal to $1$ (mod 4) and let $U=\Spec \ZZ\setminus \{p,q\}$. Let $x_p$ and $x_q$ be the elements in $H^1(U,\ZZ/2\ZZ)$ corresponding to the extensions $\QQ(\sqrt{p})$ and $\QQ(\sqrt{q})$ respectively. Then $x_p\cupp x_q$ is completely determined by the Legendre symbol $\legendre{p}{q}$ and vice versa. In particular, $x_p\cupp x_q=0$ if and only if $\legendre{p}{q}=1$.
\end{prop}

\begin{proof}
We have that $x_p\cupp x_q=0$ is zero if and only if $\langle x_p\cupp x_q,\alpha\rangle=0$ for all $\alpha\in H^1_c(\Ufl,\mmu_2)$, where
  \[
    \langle -,-\rangle \colon H^2(U,\ZZ/2\ZZ)\times H^1_c(\Ufl,\mmu_2)\to \QQ/\ZZ
  \]
is the pairing coming from Artin--Verdier duality.
Let $S=\{p,q\}\subset \Spec \ZZ$. Since $\QQ$ has trivial class group we have
  \[
    H^1_c(\Ufl,\mmu_2)\cong C_{\QQ,S}[2]\cong (\ZZ_p^\times \times\ZZ_q^\times)[2]\,,
  \]
which is generated by the elements $(-1,1)$ and $(1,-1)$.

First consider the case when $\alpha$ is the class represented by $(1,-1,1,1)\in I_{\QQ}=\RR^\times\times\QQ_p^\times\times \QQ_q^\times \times \Pi_{v \neq p,q, \infty} \QQ_v$. By Theorem \ref{thm:cup} we have
  \[
    \langle x_p\cupp x_q,\alpha\rangle = \langle x_q,\alpha N(\beta)\rangle
  \]
where $\beta$ is an element in $I_{\QQ[\sqrt{p}]}$ such that $\alpha=t\beta/\sigma(\beta)$, where $t\in \QQ[\sqrt{p}]^\times$ satisfies $N(t)=\alpha^{-n}$. Here $\sigma$ is the generator of $\Gal(\QQ[\sqrt{p}]/\QQ)$ and $N$ denotes the norm on idèles $N\colon I_{\QQ[\sqrt{p}]}\to I_\QQ$.
The group $I_{\QQ[\sqrt{p}]}$ will have one component above the prime $p$ and one or two components above $q$ depending on whether $q$ splits in $\QQ[\sqrt{p}]$ or not.
Let $\beta\in I_{\QQ[\sqrt{p}]}$ be the element with component over $p$ equal to $\sqrt{p}$ and with a 1 in all other components. Then $\beta/\sigma(\beta)$ has component $-1$ at the prime above $p$  and is equal to $1$ in all other components. We further have that  $\alpha N(\beta) \in I_\QQ$ is the idèle which is $p$ in the component corresponding to $p$ and is otherwise equal to $1.$ We thus have that $$ \langle x_p\cupp x_q,\alpha\rangle = \langle x_q,\alpha N(\beta)\rangle$$ is zero if and only if $p$ is in the image of the norm map $N\colon I_{\QQ[\sqrt{q}]}\to I_{\QQ}$, which is to say, if and only if $p$ splits in $\QQ(\sqrt{q}).$  \\

The case when $\alpha $ equals $ (1,1,-1,1) \in I_{\QQ,S}=\RR^\times\times\QQ_p^\times\times \QQ_q^\times \times \Pi_{v \neq p,q, \infty} \QQ_v$ follows analogously by graded commutativity of the cup product.
\end{proof}

\begin{remark}
We note that as a consequence of Proposition \ref{legendre}, the quadratic reciprocity law  $\legendre{p}{q} = \legendre{q}{p}$ then follows from graded commutativity of the cup product.
\end{remark}

\appendix

\newpage
\section{Computations}\label{app:comp}
Consider the big diagram (\ref{eq:big}). The differential $d_{\Gamma_c}^0$ is given by
  \[
    d_{\Gamma_c}^0=\begin{pmatrix}
      \sigma-1 \\
      n \\
      -\divis \\
      -\eta
    \end{pmatrix}
  \]
where
  \[
    (\sigma'-1)(a)=(\sigma(a_{n/d})/a_1,a_1/a_2,\dots,a_{(n/d)-1}/a_{n/d})\,.
  \]

  \begin{proof}[Proof of Lemma \ref{lem:three-isos}]
  The case of $s$ was shown in Section \ref{sec:cohomology-groups}.
  To prove that statement for $t$ we consider the spectral sequence of $R\Gamma_c(\Ufl,\HOM(C(\hat{u}),\C))$ obtained by filtering by columns of the corresponding double complex. This spectral sequence has page $E_1$ given by the double complex
    \[
      H^q(C(\Gamma(U,G(\HOM(C(\hat{u}),\C)))\to \Gamma(Z',\gamma^*G(\HOM(C(\hat{u}),\C)))\oplus \bigoplus_{\nu\in \Omega_\RR}\Gamma(K_\nu,a^\nu_*G(\HOM(C(\hat{u}),\C)_\nu)))[-1])^p
    \]
  for $q\geq 0$, which looks like
  \[\begin{sseq}{0...6}{0...6}
  \ssdropbull
  \ssarrow{1}{0}
  \ssdropbull
  \ssarrow{1}{0}
  \ssdropbull
  \ssarrow{1}{0}
  \ssdropbull
  \ssarrow{1}{0}
  \ssdropbull
  \ssmoveto{0}{2}
  \ssdropbull
  \ssarrow{1}{0}
  \ssdropbull
  \ssarrow{1}{0}
  \ssdropbull
  \ssarrow{1}{0}
  \ssdropbull
  \ssarrow{1}{0}
  \ssdropbull
  \end{sseq}\]
  from which it is obvious that the spectral sequence converges.
  The $\ZZ/n\ZZ$-torsor $Y$ is of the form $Y=\Ind_{C_d}^{C_n}(Y')$ for $Y'=\Spec \OO_L$ with $\OO_L$ the ring of integers of an extension $L/K$ unramified over $U$. We write $U_{Y'}=U\times_XY'$. The differential $d_1^{0,2}\colon E_1^{0,2}\to E_1^{1,2}$ is given by
    \[\Br(L)^{n/d}\to \begin{matrix}
      \Br(L)^{n/d} \\
      \oplus \\
      \Br(L)^{n/d} \\
      \oplus \\
      \bigoplus_{p\in U_{Y'}}\Br(L_p)^{n/d} \\
      \oplus \\
      \bigoplus_{w\in S'}\Br(L_w)^{n/d} \\
      \oplus \\
      \bigoplus_{v\in \Omega_\RR}\Br(L_v)^{n/d}
    \end{matrix}\]
  which is injective since
    \[\Br(L)^{n/d}\to \begin{matrix}
    \bigoplus_{v\in U_{Y'}}\Br(L_v)^{n/d} \\
    \oplus \\
    \bigoplus_{w\in S'}\Br(L_w)^{n/d}
    \end{matrix}\cong \bigoplus_{v\in Y'}\Br(L_v)^{n/d}\]
  is just $(\inv_v)_{v}^{n/d}$ which we know is injective (see Appendix \ref{app-coh-of-id}). Hence the $E_2$-page will look like
  \[\begin{sseq}{0...6}{0...6}
  \ssdropbull
  \ssmoveto{1}{0}
  \ssdropbull
  \ssmoveto{2}{0}
  \ssdropbull
  \ssmoveto{3}{0}
  \ssdropbull
  \ssmoveto{4}{0}
  \ssdropbull
  \ssmoveto{1}{2}
  \ssdropbull
  \ssmoveto{2}{2}
  \ssdropbull
  \ssmoveto{3}{2}
  \ssdropbull
  \ssmoveto{4}{2}
  \ssdropbull
  \end{sseq}\]
  and if we look at the filtration associated to this spectral sequence, the index $(j,2)$ does not contribute before level $3$ if $j\geq 1$.
  This shows that $t$ is an isomorphism.

  The fact that the leftmost and right most vertical zig-zags of (\ref{eq:three-isos}) gives isomorphisms on cohomology in degree 2 follows from the spectral sequence argument on page 5, together with the fact that the morphism of complexes 
    \[
      \begin{tikzcd}
        K^\times \ar{r}{d^{0}}\ar{d} & K^\times\oplus \Div U\oplus\prod_{p\in S_f} K_p^\times \oplus \prod_{\nu\in \Omega_\RR}K^\times_\nu\ar{d} \ar{r}{d^{1}} &\Div U\oplus \prod_{p\in S_f} K_p^\times \oplus \prod_{\nu\in \Omega_\RR}K^\times_\nu\ar{d} \\
        K^\times \ar{r}{d^{0}} & K^\times\oplus \Div U\oplus\prod_{p\in S_f} K_p^\times \oplus \prod_{\nu\in \Omega_\RR}K^\times_\nu/2K^\times_\nu \ar{r}{d^{1}} & \Div U\oplus \prod_{p\in S_f} K_p^\times \oplus \prod_{\nu\in \Omega_\RR}K^\times_\nu/2K^\times_\nu \,,
      \end{tikzcd}  
    \]
    where
    \[
        d^{0} =
        \begin{pmatrix}
          -n \\
          \divis \\
          \eta \\
          (\iota_\nu)_\nu
        \end{pmatrix}\mbox{ and }
        d^{1} =
        \begin{pmatrix}
        \divis & n & 0 & 0 \\
        \eta & 0 & n & 0 \\
        (\iota_\nu)_\nu & 0 & 0 & n
        \end{pmatrix}\,,
    \]
  gives isomorphisms on $H^1$ and $H^2$. 

  The statement about $t'$ and $\Gamma_c(q(\hat{u})^*)$ follows by the 2-out-of 3 property of isomorphisms since the functors $R\Gamma_c(\Ufl, -)$ and $R\hat{\Gamma}_c(\Ufl, -)$ preserve quasi-isomorphisms. 
  \end{proof}

\section{Cohomology of idèles}\label{app-coh-of-id}
Most results of the following section are standard and can be found in most books on class field theory. Our main references is \cite{Cassels--Frohlich}.
Let $L/K$ be a Galois extension of number fields with Galois group $G$. For every prime $v$ in $K$ we let $L^v=L_w$ for some (any) prime $w$ over $v$ and we write $G^v=\Gal(L^v/K_v)$.

\begin{lem}[{\cite[VII.7.3.(b)]{Cassels--Frohlich}}]
  We have an isomorphism
    \[H^{i}(G,I_L)\cong \bigoplus_{v}H^{i}(G^v,(L^v)^\times)\]
  for every $i\in \ZZ$.
\end{lem}

If we consider the cohomology sequence associated to the short exact sequence
  \[0\to L^\times\to I_L\to C_L\to  0\]
and using the fact that $H^1(G,C_L)=0$ \cite[VII.9.1.(2)]{Cassels--Frohlich} we get an injection
  \[H^2(G,L^\times)\to \bigoplus_{v}H^{i}(G^v,(L^v)^\times)\,.\]

If we consider the limit over all finite sub-extensions $K\subseteq L\subseteq K^{sep}$, we get the following result:

\begin{prop}
  We have a canonical injection
    \[\Br(K)\to \bigoplus_{v}\Br(K_v)\,.\]
\end{prop}

\begin{proof}
  Taking the limit over all finite sub-extensions $K\subseteq L\subseteq K^{sep}$, we get an injection
    \[H^2(\Gal(K^{sep}/K),(K^{sep})^\times)\to H^2(\Gal(K^{sep}/K),I_{K^{sep}})\]
  and the result follows since $\Br(K)\cong H^2(\Gal(K^{sep}/K),(K^{sep})^\times)$ and
    \[\begin{split}
    H^2(\Gal(K^{sep}/K),I_{K^{sep}})
    & \cong \bigoplus_{v}H^{i}(\Gal((K^{sep})^v/K_v),((K^{sep})^v)^\times) \\
    & \cong \bigoplus_{v}H^{i}(\Gal((K^{sep})^v/K_v),((K_v)^{sep})^\times) \\
    & \cong \bigoplus_{v}\Br(K_v)\,.
    \qedhere\end{split}\]
\end{proof}

\bibliographystyle{dary}
\bibliography{cupproducts}{}
\end{document}